\documentclass[reqno]{amsart}

\usepackage{tabu}
\usepackage{amssymb}
\usepackage{mathtools}
\usepackage{a4wide,amsmath}
\usepackage{mathrsfs}
\usepackage{amsthm}
\numberwithin{equation}{section}
\numberwithin{figure}{section}
\numberwithin{table}{section}
\usepackage{bbm}
\usepackage{subfig}
\usepackage{enumerate}
\usepackage[section]{placeins}
\usepackage{graphicx}		  
\usepackage{ifpdf}
\ifpdf
\DeclareGraphicsExtensions{.pdf,.eps,.jpg,.png}	
\usepackage[suffix=]{epstopdf}
\fi
\usepackage{xcolor}
\usepackage[utf8]{inputenc}
\usepackage{hyperref}
\hypersetup{hidelinks}

\usepackage{caption}
\long\def\MSC#1\EndMSC{\def\arg{#1}\ifx\arg\empty\relax\else
	{\narrower\noindent%
		{2020 Mathematics Subject Classification}: #1\\} \fi}
\long\def\PACS#1\EndPACS{\def\arg{#1}\ifx\arg\empty\relax\else
	{\narrower\noindent%
		{PACS numbers}: #1}\fi}
\long\def\KEY#1\EndKEY{\def\arg{#1}\ifx\arg\empty\relax\else
	{\narrower\noindent%
		Keywords: #1\\}\fi}

\theoremstyle{plain}
\newtheorem{theorem}{Theorem}[section]
\newtheorem{lemma}[theorem]{Lemma}

\newtheorem{corollary}[theorem]{Corollary}
\theoremstyle{definition}

\newtheorem{example}[theorem]{Example}
\newtheorem{assumption}[theorem]{Assumption}
\theoremstyle{remark}
\newtheorem{remark}[theorem]{Remark}

\newcommand{\norm}[1]{\lVert#1\rVert}
\newcommand{\abs}[1]{\lvert#1\rvert} 
\newcommand{\inner}[1]{\langle#1\rangle}

\newcommand{\redel}{\mathop{\textup{Re}}}

\newcommand{\mspan}{\mathop{\textup{span}}}

\newcommand{\essinf}{\mathop{\textup{ess\,inf}}}
\newcommand{\esssup}{\mathop{\textup{ess\,sup}}}
\newcommand{\ident}{\mathop{\textup{id}}}

\newcommand{\di}{\mathrm{d}}   

\newcommand{\R}{\mathbb{R}}
\newcommand{\N}{\mathbb{N}}
\newcommand{\C}{\mathbb{C}}
\newcommand{\Z}{\mathbb{Z}}

\begin{document}
	\title[Series reversion in Calder{\'o}n's problem]{Series reversion in Calder{\'o}n's problem}
	
	\author[H.~Garde]{Henrik Garde}
	\address[H.~Garde]{Department of Mathematics, Aarhus University, Ny Munkegade 118, 8000 Aarhus C, Denmark.}
	\email{garde@math.au.dk}
	
	\author[N.~Hyv\"onen]{Nuutti Hyv\"onen}
	\address[N.~Hyv\"onen]{Department of Mathematics and Systems Analysis, Aalto University, P.O. Box~11100, 00076 Helsinki, Finland.}
	\email{nuutti.hyvonen@aalto.fi}
	
	\begin{abstract}
	  This work derives explicit series reversions for the solution of Calder\'on's problem. The governing elliptic partial differential equation is $\nabla\cdot(A\nabla u)=0$ in a bounded Lipschitz domain and with a matrix-valued coefficient. The corresponding forward map sends $A$ to a projected version of a local Neumann-to-Dirichlet operator, allowing for the use of partial boundary data and finitely many measurements. It is first shown that the forward map is analytic, and subsequently reversions of its Taylor series up to specified orders lead to a family of numerical methods for solving the inverse problem with increasing accuracy. The convergence of these methods is shown under conditions that ensure the invertibility of the Fr\'echet derivative of the forward map. The introduced numerical methods are of the same computational complexity as solving the linearised inverse problem. The analogous results are also presented for the smoothened complete electrode model.
	\end{abstract}	
	\maketitle
	
	\KEY
	Calder\'on problem,
	electrical impedance tomography, 
	series reversion.
	\EndKEY
	
	\MSC
	35R30,
	41A58,
	47H14.
	\EndMSC
	
	\section{Introduction} \label{sec:intro}
	
	Let $\Omega$ be a bounded Lipschitz domain in $\R^d$, $d\geq 2$. Calder\'on's inverse conductivity problem consists in determining the coefficient $A$ in the generalised Laplace equation
	\begin{equation}
          \label{eq:gen_lapla}
		-\nabla\cdot (A\nabla u) = 0 \text{ in } \Omega
	\end{equation}
	from boundary measurements, i.e.~from Cauchy data of solutions to \eqref{eq:gen_lapla}. In this work, idealised boundary measurements are modelled by a local {\em Neumann-to-Dirichlet} (ND) map $\Lambda(A)$ that may be defined on an arbitrarily small relatively open subset $\Gamma$ of $\partial \Omega$. The considered coefficient $A$ is allowed to be anisotropic and complex-valued; more precisely, $A$ is assumed to be an element of
	\begin{equation*}
		L^\infty_+(\Omega) = \Bigl\{ A\in L^\infty(\Omega;\mathbb{C}^{d\times d}) \bigm| \exists c_A > 0,\, \forall \xi\in \mathbb{C}^d : \essinf_{x\in\Omega}\,\redel[(A(x)\xi)\cdot \overline{\xi}\,] \geq c_A \abs{\xi}^2 \Bigr\}.
	\end{equation*}
    The main result of this work is an asymptotic formula for the solution of Calder\'on's problem, which leads to a family of numerical methods of arbitrarily high order for reconstructing an additive perturbation $B\in L^\infty(\Omega; \mathbb{C}^{d\times d})$ to a known coefficient $A\in L^\infty_+(\Omega)$ from (partial) knowledge of $\Lambda(A + B)$.

    Let us be more precise. Based on the reversion of the Taylor series for the analytic forward map $A \mapsto \Lambda(A)$ (cf.~\cite{Calderon1980,Garde_2019c}), we prove an explicit asymptotic formula for reconstructing $B$,
	\begin{equation}
          \label{eq:intro_expansion}
		B = \sum_{j=1}^K F_j + O(\norm{B}^{K+1}), \qquad K \in \N,
	\end{equation}
	where $F_j = O(\norm{B}^j)$ are solely based on $A$ and $\mathscr{P}\Lambda(A+B)\mathscr{P}$, and $\mathscr{P}$ can be chosen as the orthogonal projection onto any closed mean free subspace of $L^2(\Gamma)$, subject to the following conditions. For \eqref{eq:intro_expansion} to be valid, the projected Fr\'echet derivative $\mathscr{P}D\!\Lambda(A;\,\cdot\,)\mathscr{P}$ must be injective on a closed subspace $\mathcal{W} \subset  L^\infty(\Omega; \mathbb{C}^{d\times d})$, known \emph{a priori} to contain $B$, and it must also map $\mathcal{W}$ onto a closed complemented subspace in a suitable space of linear operators. Under these assumptions there exists a projection $Q$, acting on an appropriate Banach space of linear operators and ensuring the compatibility with $\mathcal{W}$, such that the mapping
	\begin{equation*}
		B\mapsto Q(\mathscr{P}\Lambda(A+B)\mathscr{P}-\mathscr{P}\Lambda(A)\mathscr{P})
	\end{equation*}
	has an analytic inverse for small enough $B$, and \eqref{eq:intro_expansion} is its truncated Taylor series. For details on the projection $Q$ we refer to Section~\ref{sec:reversion} and Appendix~\ref{sec:HS}, as well as to the implementation details in Section~\ref{sec:implementation} on how to avoid explicitly forming $Q$ if $\mathcal{W}$ is finite-dimensional. For completeness, it should be mentioned that a related series reversion approach based on the Born series has previously been considered for Calder\'on's problem in~\cite{Arridge_2012}.
	
	The computational complexity of a numerical implementation of \eqref{eq:intro_expansion} for any fixed $K \in \N$ is of the same order as that of solving the corresponding linearised inverse problem. That is, the number of required floating point operations is bounded by a  $K$-dependent constant times the number of operations needed for solving the linearised problem, independently of the employed level of discretisation for~\eqref{eq:gen_lapla}. Furthermore, all ill-conditioned steps in an implementation of \eqref{eq:intro_expansion} correspond to inverting the first derivative $\mathscr{P}D\!\Lambda(A;\,\cdot\,)\mathscr{P}$, and thus any regularisation method designed for linear inverse problems can be used in connection to \eqref{eq:intro_expansion}.

    The imaging modality that corresponds to Calder\'on's problem is {\em electrical impedance tomography} (EIT), where an electrical conductivity distribution is reconstructed from boundary measurements; see the review articles \cite{Borcea2002a,Borcea2002,Cheney1999} and the references therein for more information on EIT. Modelling measurements of EIT by a local ND map corresponds to the {\em continuum model} (CM) of EIT with partial data. In addition to considering the CM, we also present analogous series reversion results for the {\em smoothened complete electrode model} (SCEM) \cite{Hyvonen2017}, which is a generalisation of the standard {\em complete electrode model} (CEM) of EIT~\cite{Cheng89,Somersalo1992} with the potential for more efficient numerical solution. In particular, all our results also apply to the standard CEM which is capable of modelling EIT up to the measurement precision. The analyticity of the forward map of the CEM, a main tool in our analysis, has previously been considered in~\cite{GardeStaboulis_2016}. 

    If $\mathcal{W}$ is finite-dimensional, the injectivity of $\mathscr{P}D\!\Lambda(A;\,\cdot\,)\mathscr{P}$ on $\mathcal{W}$ can often be guaranteed by assuming a high enough number of boundary measurements compared to the dimension of $\mathcal{W}$. See~\cite{Alberti2021} for such a result for the CM and \cite{Lechleiter2008} for related analysis in the framework of the CEM. More generally, the unique and Lipschitz-stable solution of the {\em nonlinear} Calder\'on problem, with a finite-dimensional parametrisation of the unknown coefficient, has previously been considered in \cite{Alberti2019,Alberti2020} for the CM and in \cite{Harrach_2019} for the CEM.

    Let us briefly review some global uniqueness results for Calder\'on's problem; see the review papers \cite{Kenig_2014,Uhlmann2009} and the references therein for more information. For smooth enough isotropic real coefficients and complete boundary measurements ($\Gamma = \partial \Omega$), the global uniqueness was shown for $d\geq 3$ in \cite{Sylvester1987} and for $d=2$ in \cite{Nachman1996}; the unique identifiability of piecewise analytic coefficients was established already in \cite{Kohn1985}. The regularity assumptions on the coefficient have since been reduced to Lipschitz continuity for $d \geq 3$~\cite{CaroRogers2016} and $L^\infty$ for $d=2$~\cite{Astala2006a}. 
    
    There are numerous uniqueness results for the isotropic Calder\'on problem with partial data (see,~e.g.,~\cite{Ferreira2009,Imanuvilov2010,Imanuvilov_2015,Isakov2007,Kenig_2013,Kohn1985}). As the partial data case for the CM is in the focus of our attention, it should be mentioned that there are also previous reconstruction algorithms suitable for such a task; see,~e.g.,~\cite{Garde_2019b,Nachman2010} as well as,~e.g.,~\cite{Garde2020c,Hanke03,Harrach10,Harrach13} for the more specific task of detecting inclusions.
    
    It is well-known that Calder\'on's problem is not uniquely solvable for anisotropic coefficients in general. However, in two dimensions it has been shown that an anisotropic $L^\infty$-coefficient is uniquely determined up to a pushforward by an $H^1$-diffeomorphism that fixes the boundary \cite{Astala2005,Sylvester1990}. Moreover, by sufficiently restricting the considered class of anisotropic coefficients, there are actually examples of unique identification \cite{Alessandrini2017,Alessandrini2018}.

    Since it is possible to approximate CM measurements by those of the CEM if the number of electrodes tends to infinity and the electrodes cover the object boundary in a controlled manner~\cite{GardeHyvonen2021,Hyvonen09}, some of the aforementioned global uniqueness results on Calder\'on's problem can be transferred to the framework of the CEM in a sense of limits.
        
    This article is organised as follows. Section~\ref{sec:continuum} recalls the CM for matrix-valued coefficients, and Section~\ref{sec:taylor} introduces a Taylor series representation for the associated forward map. Sections \ref{sec:SCEM} and~\ref{sec:taylor_SCEM} provide the analogous analysis for the SCEM; readers not interested in electrode models of EIT can skip these sections. Our main results are presented in Section~\ref{sec:reversion}, where the recursive technique for inverting a suitably restricted relative forward map of the CM or the SCEM is presented. Section~\ref{sec:implementation} considers efficient implementation of the introduced family of numerical methods for approximately solving Calder{\'o}n's problem, and it also provides a couple of numerical examples. Appendix~\ref{sec:HS} employs a Hilbert--Schmidt structure in two spatial dimensions for systematic selection of the projection $Q$ needed in our analysis.
	
	\subsection{Some notational remarks} \label{sec:notation}
	
	$\mathscr{L}(X,Y)$ is the space of bounded linear operators between Banach spaces $X$ and $Y$, with the shorthand notation $\mathscr{L}(X) = \mathscr{L}(X,X)$. The corresponding spaces of compact operators are denoted by $\mathscr{L}_\textup{C}(X,Y)$ and $\mathscr{L}_\textup{C}(X)$. More generally, we denote the space of bounded $k$-linear maps from $X^k$ to $Y$ by $\mathscr{L}^k(X,Y)$, and equip it with the norm
	\begin{equation*}
		\norm{F}_{\mathscr{L}^k(X,Y)} = \sup\bigl\{ \norm{F(x_1,\dots,x_k)}_Y \mid \norm{x_j}_X \leq 1,\enskip j=1,\dots,k \bigr\}.
	\end{equation*}
	
	The Euclidean inner product on $\mathbb{C}^n$ is denoted $x\cdot\overline{y}$ for $x,y\in\mathbb{C}^n$. In particular, the ``dot'' is used as a bilinear mapping. The Euclidean norm of $x\in\mathbb{C}^n$ is denoted $\abs{x}$.
	
	Our analysis considers several operators that depend on a parameter as well as maps into spaces of operators. To allow a readable notation, we often separate variables of different natures by semicolons. As an example, we write $D^2\!\Lambda(A;B_1,B_2)f$, instead of $(D^2\!\Lambda(A)(B_1,B_2))f$, for the second derivative of the forward map $\Lambda$ at $A$ evaluated in directions $B_1$ and $B_2$ and operating on a Neumann boundary value $f$. In particular, note that $D^2\!\Lambda(A;B_1,B_2)f$ depends nonlinearly on $A$ but linearly on $B_1$, $B_2$, and $f$. When arguments are suppressed, we often use the notation $D^2\!\Lambda(A)$ instead of $D^2\!\Lambda(A; \,\cdot\,, \,\cdot\,)$.
	
	\section{Continuum model} \label{sec:continuum}
	
	Let $\Omega\subset \mathbb{R}^d$,  $d\in\mathbb{N}\setminus\{1\}$, be a bounded Lipschitz domain and let $\Gamma\subseteq \partial\Omega$ be relatively open. We define a norm on $L^\infty(\Omega; \mathbb{C}^{d\times d})$ via
	\begin{equation*}
		\norm{B} = \esssup_{x\in\Omega} \norm{B(x)}_2, \quad B\in L^\infty(\Omega; \mathbb{C}^{d\times d}),
	\end{equation*}
	where $\norm{\, \cdot \,}_2$ denotes the standard spectral norm. For $B\in L^\infty(\Omega; \mathbb{C}^{d\times d})$ and almost all $x\in\Omega$, it obviously holds
	\begin{equation*} 
		\abs{B(x)\xi} \leq \norm{B}\abs{\xi}, \quad \xi\in\mathbb{C}^d.
	\end{equation*}
	Moreover,
	\begin{equation*}
		\inner{w,v}_B = \int_{\Omega} (B\nabla w)\cdot \overline{\nabla v}\,\di x
	\end{equation*}
	defines a continuous sesquilinear form on $H^1(\Omega)$ and on its subspace
	\begin{equation*}
		H^1_\diamond(\Omega) = \{ w\in H^1(\Omega) \mid \inner{1,Tw}_{L^2(\Gamma)} = 0 \}.
	\end{equation*}
    Here $T: H^1(\Omega) \to L^2(\Gamma)$ is the Dirichlet trace operator onto $\Gamma$. We equip $H_\diamond^1(\Omega)$ with the norm
	\begin{equation*}
		\norm{w}_{*}^2 = \int_\Omega \abs{\nabla w}^2\,\di x,
	\end{equation*}
	which is equivalent to the standard $H^1(\Omega)$-norm on $H_\diamond^1(\Omega)$ by virtue of a Poincar\'e inequality:
	\begin{equation*}
		\norm{w}_{*}^2\leq \norm{w}_{H^1(\Omega)}^2 \leq C_{\textup{P}} \norm{w}_{*}^2, \quad w\in H^1_{\diamond}(\Omega).
	\end{equation*}
	
	If $A\in L^\infty_+(\Omega)$, then also $A+B \in L^\infty_+(\Omega)$ for any $B\in L^\infty(\Omega; \mathbb{C}^{d\times d})$ with $\norm{B}<c_A$, where $c_{A+B} = c_A-\norm{B}$ may be used. The following continuity and coercivity estimates hold for all $w,v\in H^1_\diamond(\Omega)$:
	\begin{align}
		\abs{\inner{w,v}_A} &\leq \norm{A}\norm{w}_*\norm{v}_*, \notag\\[1mm]
		\abs{\inner{v,v}_A} &\geq \redel\inner{v,v}_A \geq c_A\norm{v}_*^2, \label{eq:coercive1}
	\end{align}
as easily deduced from the above definitions.
	
	The CM corresponds to the following elliptic boundary value problem with a coefficient $A\in L^\infty_+(\Omega)$:
	\begin{align*}
		-\nabla\cdot(A\nabla u) &= 0 \text{ in } \Omega, \\[1mm]
		\nu\cdot(A\nabla u) &= \begin{cases}
			f & \text{ on } \Gamma, \\
			0 & \text{ on } \partial\Omega\setminus \overline{\Gamma}.
		\end{cases} 
	\end{align*}
	Here, $\nu$ is the exterior unit normal of $\partial \Omega$, and the Neumann boundary value $f$ belongs to the $\Gamma$-mean free space
	\begin{equation*}
		L_\diamond^2(\Gamma) = \{ w\in L^2(\Gamma) \mid \inner{1,w}_{L^2(\Gamma)} = 0 \}.	
	\end{equation*}
	
	The weak form for the CM is
	\begin{equation} \label{eq:weakform}
		\inner{u,v}_A = \inner{f,Tv}_{L^2(\Gamma)}, \quad \forall v\in H^1(\Omega).
	\end{equation}
	Due to the Lax--Milgram lemma, \eqref{eq:weakform} has a unique solution $u$ in $H^1_\diamond(\Omega)$ satisfying the bound
	\begin{equation}
		\norm{u}_{*} \leq \frac{C_{\textup{T}}}{c_A}\norm{f}_{L^2(\Gamma)}, \label{eq:ubound}
	\end{equation}
	where $C_{\textup{T}} = \norm{T}_{\mathscr{L}(H_\diamond^1(\Omega),L_\diamond^2(\Gamma))}$. We occasionally write $u = u_f^A$ in order to be specific about the connection of $u$ to $A$ and $f$.
	
	Let us then introduce three mappings related to the CM:
	\begin{enumerate}[(i)]
		\item $N : L^\infty_+(\Omega) \to \mathscr{L}(L^2_\diamond(\Gamma),H^1_\diamond(\Omega))$ defined by $N(A)f = u_f^A$.
		\item $\Lambda : L^\infty_+(\Omega) \to \mathscr{L}_\textup{C}(L^2_\diamond(\Gamma))$ defined by $\Lambda(A)f = TN(A)f = Tu_f^A$.
		\item $P_A \in \mathscr{L}(L^\infty(\Omega;\mathbb{C}^{d\times d}),\mathscr{L}(H^1_\diamond(\Omega)))$ for $A\in L^\infty_+(\Omega)$ defined by
		\begin{equation} \label{eq:defP}
			\inner{P_A(B)y ,v}_A = -\inner{y,v}_B, \quad \forall v\in H^1(\Omega),
		\end{equation}
		where $B \in L^\infty(\Omega;\mathbb{C}^{d\times d})$ and $y \in H^1_\diamond(\Omega)$.
	\end{enumerate}
    The operator $P_A$ is well-defined due to the Lax--Milgram lemma that guarantees the unique solvability of the variational problem \eqref{eq:defP} and also yields the estimate
    \begin{equation}
    	\norm{P_A(B)}_{\mathscr{L}(H^1_\diamond(\Omega))} \leq \frac{\norm{B}}{c_A}. \label{eq:Pbound}
    \end{equation}
    The compactness of $\Lambda(A)$ is a well-known consequence of compact embeddings between Sobolev spaces on $\partial \Omega$; cf.~Remark~\ref{rm:quotient}.

	For a given $A\in L^\infty_+(\Omega)$, the linear operator $\Lambda(A)$ is called the local ND map on the boundary piece $\Gamma$, while the nonlinear map $\Lambda$ is the forward map of the partial data Calder\'on problem for the CM.
	
	\section{Taylor series for the CM} \label{sec:taylor}
	
	We start by deriving a Taylor series for $N$ and the forward map $\Lambda$. These series have a maximal radius of convergence in the sense that they converge for all perturbations $B\in L^\infty(\Omega; \mathbb{C}^{d\times d})$ with $\norm{B} < c_A$, which is the natural condition for ensuring $A+B\in L^\infty_+(\Omega)$. Note that the proofs in \cite[Appendix~A]{Garde_2019c} can be directly adapted to our current setting that allows more general coefficients and local ND maps. Be that as it may, we present below the key ideas of the proofs for the sake of completeness.
	
	\begin{lemma}
		\label{lemma:perturbop}
		$P_A(B)$, and more generally $P_A$, is infinitely times continuously Fr\'echet differentiable with respect to $A$. Its first derivative $D_A P_A(B)$ is given by
		\begin{equation*}
			D_A P_A(B;C) = P_A(C)P_A(B)
		\end{equation*}
		for $A \in L^\infty_+(\Omega)$ and $B, C \in L^\infty(\Omega; \mathbb{C}^{d\times d})$.
	\end{lemma}
	\begin{proof}
		Let $\norm{C}$ be small enough so that $A+C\in L^\infty_+(\Omega)$. According to the definitions of $P_{A+C}(B)$ and $P_A(B)$, we have
		\begin{equation*}
			\inner{P_{A+C}(B)y,v}_{A+C} = -\inner{y,v}_B = \inner{P_A(B)y,v}_A,\quad \forall v\in H^1(\Omega).
		\end{equation*}
		Restructuring and using the definition of $P_A(C)$ gives
		\begin{equation*}
			\big \langle (P_{A+C}(B) - P_A(B))y , v \big\rangle_A = -\inner{P_{A+C}(B)y,v}_C = \inner{P_A(C)P_{A+C}(B)y,v}_A
		\end{equation*}
		for all $y\in H^1_\diamond(\Omega)$, $v\in H^1(\Omega)$, and $B\in L^\infty(\Omega; \mathbb{C}^{d\times d})$. In consequence,
		\begin{equation}
			P_{A+C}(B) - P_A(B) = P_A(C)P_{A+C}(B), \label{eq:Pequality}
		\end{equation}
		and thus \eqref{eq:Pbound} implies
		\begin{equation}
			\norm{P_{A+C}(B) - P_A(B)}_{\mathscr{L}(H^1_\diamond(\Omega))} \leq \frac{\norm{B}\norm{C}}{c_A c_{A+C}}. \label{eq:Pbound2}
		\end{equation}
   		The assertion about the first derivative of $P_A(B)$ with respect to $A$ now follows by applying \eqref{eq:Pequality}, \eqref{eq:Pbound}, and \eqref{eq:Pbound2} to deduce
		\begin{align*}
			\norm{P_{A+C}(B) - P_A(B) - P_A(C)P_A(B)}_{\mathscr{L}(H^1_\diamond(\Omega))} &= \norm{P_A(C)(P_{A+C}(B)-P_A(B))}_{\mathscr{L}(H^1_\diamond(\Omega))} \\[1mm]
			&\leq \frac{\norm{B}\norm{C}^2}{{c_A}^2 c_{A+C}} = o(\norm{C})\norm{B}.
		\end{align*}
		Moreover, due to the norm of $B$ on the right, we also immediately obtain the differentiability of $P_A$ with respect to $A$. Indeed, taking the operator norm on $\mathscr{L}(L^\infty(\Omega;\mathbb{C}^{d\times d}),\mathscr{L}(H^1_\diamond(\Omega)))$ leads to
		\begin{align*}
		\norm{P_{A+C} - P_A - P_A(C)P_A(\,\cdot\,)}_{\mathscr{L}(L^\infty(\Omega;\mathbb{C}^{d\times d}),\mathscr{L}(H^1_\diamond(\Omega)))} = o(\norm{C}).
		\end{align*}
		Finally, the product rule guarantees that $A \mapsto P_A$ is in fact infinitely times continuously differentiable.
	\end{proof}
	\begin{remark}
		According to Lemma~\ref{lemma:perturbop}, the commutator of $P_A(B)$ and $P_A(C)$ satisfies
		\begin{equation*}
			P_A(B)P_A(C) - P_A(C)P_A(B) = D_AP_A(C;B) - D_AP_A(B;C).
		\end{equation*}
		The mappings $P_A(B)$ and $P_A(C)$ do not commute and $D_AP_A(C;B)$ and  $D_AP_A(B;C)$ do not coincide in general. To illuminate this by an example, let $\Omega$ be the unit ball in $\mathbb{R}^d$ and $\Gamma = \partial\Omega$ the unit sphere, let $I$ be the identity matrix and $J$ any permutation matrix in $\mathbb{R}^{d\times d}$, and finally let $y(x) = a\cdot x$ for $x\in\Omega$ and some $a\in\mathbb{R}^d$. Then $y\in H^1_\diamond(\Omega)$, and it follows straightforwardly that 
		\begin{equation*}
			P_I(J)y = -(Ja)\cdot x = -y\circ J^\textup{T}.
		\end{equation*}
		Hence, $P_I(J_1)P_I(J_2)y = y\circ (J_1 J_2)^\textup{T}$ for any two permutation matrices $J_1$ and $J_2$. In particular, $P_I(J_1)$ and $P_I(J_2)$ can obviously fail to commute.
	\end{remark}

	Let $\rho_k$ be the collection of all permutations of indices up to $k\in \mathbb{N}$, i.e.
	\begin{equation*}
		\rho_k = \{(\alpha_1, \ldots, \alpha_k) \mid \alpha_i \in \{1, \ldots, k\} \ \text{and} \ \alpha_i \neq \alpha_j \ \text{if} \ i \neq j\}.
	\end{equation*}
	
	\begin{theorem} \label{thm:NDdiff}
		The mappings $N$ and $\Lambda$ are infinitely times continuously Fr\'echet differentiable. Their derivatives at $A \in L^\infty_+(\Omega)$ are given by $D^k \! \Lambda(A) = TD^k \! N(A)$ and
		\begin{equation}
			D^k \! N(A; B_1, \ldots, B_k) = \Bigl[ \sum_{\alpha \in \rho_k} P_A(B_{\alpha_1})\cdots P_A(B_{\alpha_k}) \Bigr]N(A) \label{eq:Nderiv}
		\end{equation}
		for $k\in\mathbb{N}$ and $B_1, \dots, B_k \in L^\infty(\Omega; \mathbb{C}^{d\times d})$. These mappings are analytic with Taylor series
		\begin{align}
			N(A + B) &= \sum_{k = 0}^\infty \frac{1}{k!} D^k\! N(A; B, \ldots, B) = \sum_{k=0}^\infty P_A(B)^k N(A), \label{eq:Nseries} \\
			\Lambda(A + B) &= \sum_{k = 0}^\infty \frac{1}{k!} D^k\! \Lambda(A; B, \ldots, B) = T\sum_{k=0}^\infty P_A(B)^k N(A) \label{eq:Lseries}
		\end{align}
		for $A \in L^\infty_+(\Omega)$ and $B \in L^\infty(\Omega; \mathbb{C}^{d\times d})$ such that $\norm{B} < c_A$. 
	\end{theorem}
	\begin{proof}
		As in \cite[Proof of Theorem~A.2]{Garde_2019c}, it is sufficient to prove the expansion \eqref{eq:Nseries}. Note that \eqref{eq:Lseries} follows directly from \eqref{eq:Nseries} and the definition of $\Lambda$. 
		
		Let $A\in L^\infty_+(\Omega)$ and $B\in L^\infty(\Omega; \mathbb{C}^{d\times d})$, with $\norm{B}$ small enough so that also $A+B\in L^\infty_+(\Omega)$. The definitions of $N(A+B)$ and $N(A)$ yield
		\begin{equation}
			\inner{N(A+B)f,v}_{A+B} = \inner{f,Tv}_{L^2(\Gamma)} = \inner{N(A)f,v}_A \label{eq:Ntmp}
		\end{equation}
		for all $f\in L^2_\diamond(\Gamma)$ and $v\in H^1(\Omega)$. Combining the definition of $P_A(B)$ with \eqref{eq:Ntmp} gives
		\begin{align*}
			\inner{P_A(B)N(A+B)f,v}_A &= -\inner{N(A+B)f,v}_B \\[0.5mm]
			&= \inner{N(A+B)f,v}_{A+B} - \inner{N(A)f,v}_A -\inner{N(A+B)f,v}_B \\[0.5mm]
			&= \inner{(N(A+B)-N(A))f,v}_A
		\end{align*}
		for all $f\in L^2_\diamond(\Gamma)$ and $v\in H^1(\Omega)$. Due to \eqref{eq:coercive1}, we have actually proven the equality
		\begin{equation*}
			P_A(B)N(A+B) = N(A+B)-N(A),
		\end{equation*}
		which may also be written as
		\begin{equation*}
			(\ident-P_A(B))N(A+B) = N(A).
		\end{equation*}

                Assume that $\norm{B} < c_A$, which guarantees $A+B \in L^\infty_+(\Omega)$ as well as $\norm{P_A(B)}_{\mathscr{L}(H^1_\diamond(\Omega))} < 1$ by virtue of \eqref{eq:Pbound}. The latter allows inverting $\ident-P_A(B)$ via a Neumann series, which gives the sought-for expansion in \eqref{eq:Nseries}. The actual formulas for the derivatives of $N$ can then be deduced inductively from \eqref{eq:Nseries}, Lemma~\ref{lemma:perturbop}, and the product rule as in \cite[Proof of Theorem~A.2]{Garde_2019c}.
	\end{proof}
	
	If $A\in L^\infty_+(\Omega)$ is Hermitian, it is straightforward to derive a more standard formula for the first Fr\'echet derivative of $\Lambda$. Obviously, $D\!\Lambda(A;B)g = TP_A(B)u_g^A$, which by \eqref{eq:weakform} and the definition of $P_A(B)$ gives
	\begin{equation} \label{eq:DLambdasimple}
		\inner{D\!\Lambda(A;B)g,f}_{L^2(\Gamma)} = \inner{P_A(B)u_g^A, u_f^A}_A = -\inner{u_g^A,u_f^A}_B, \quad f,g\in L^2_\diamond(\Gamma).
	\end{equation}

    The following bounds are immediate consequences of \eqref{eq:ubound}, \eqref{eq:Pbound}, and \eqref{eq:Nderiv}. See Section~\ref{sec:notation} for the norm on bounded $k$-linear maps.
	
	\begin{corollary} \label{col:cm}
	Let $A\in L^\infty_+(\Omega)$ and $k\in\mathbb{N}$. Then,
	\begin{align*}
		\norm{ D^k\! N(A) }_{\mathscr{L}^k(L^\infty(\Omega; \mathbb{C}^{d\times d}), \mathscr{L}(L^2_\diamond(\Gamma),H^1_\diamond(\Omega)))} &\leq \frac{k!C_{\textup{T}}}{{c_A}^{k+1}}, \\[1mm]
		\norm{ D^k\! \Lambda(A) }_{\mathscr{L}^k(L^\infty(\Omega; \mathbb{C}^{d\times d}), \mathscr{L}(L^2_\diamond(\Gamma)))} &\leq \frac{k!{C_{\textup{T}}}^2}{{c_A}^{k+1}}.
	\end{align*}
	\end{corollary}
	
	\begin{remark}
          \label{rm:quotient}
	  The results presented in this section could also be formulated with some other, closely related, spaces acting as the domains and/or codomains for the introduced operators. For instance, the operator $P_A(B)$ could be defined on $H^1(\Omega)/\mathbb{C}$, if that same space were also used as the codomain of $N(A)$. In a similar fashion, the codomain of $\Lambda(A)$ could be changed to $L^2(\Gamma)/\mathbb{C}$. In general, the use of such quotient spaces would explicitly emphasise the freedom in the choice of the ground level of potential, whereas the chosen framework of $\Gamma$-mean free spaces corresponds to a certain systematic way of selecting the ground level.
		
	  If $\Gamma = \partial\Omega$, one may naturally interpret $\Lambda(A) : H^{-1/2}_\diamond(\partial\Omega) \to H^{1/2}_\diamond(\partial\Omega)$, with
      \begin{equation} \label{eq:Hsdiamond}
      	H^{s}_\diamond(\partial\Omega) = \{ g \in H^{s}_\diamond(\partial\Omega) \mid \inner{1,g}_{\partial \Omega} = 0 \}, \quad -\tfrac{1}{2} \leq s \leq \tfrac{1}{2},
      \end{equation}
      where $\inner{\,\cdot\, , \,\cdot\,}_{\partial \Omega}: H^{-s}(\partial\Omega) \times H^{s}(\partial\Omega)\to\C$ denotes the sesquilinear dual bracket on $\partial \Omega$. The above results still hold in such a setting if $H^{-1/2}_\diamond(\partial\Omega)$ is selected as the domain of $N(A)$  and  the codomain of $T$ is interpreted as $H^{1/2}(\partial \Omega)$. What is more, the codomain of $\Lambda(A)$ may be chosen as $H^{1/2}(\partial\Omega)/\mathbb{C}$, if specifying the ground level of potential is for some reason unwanted. 
      
      One could also make more exact use of the Sobolev scale in the partial data setting of $\Gamma \not= \partial \Omega$ by resorting to suitable standard variants of $H^{\pm 1/2}(\Gamma)$; see,~e.g.,~\cite{Fernandes1997} for more information. However, we do not stress this matter any further in this work.       
	\end{remark}
	
	\section{Smoothened complete electrode model} \label{sec:SCEM}
	
	In this section we consider the SCEM~\cite{Hyvonen2017} that is a generalisation of the standard CEM~\cite{Cheng89}. Although \cite{Hyvonen2017} shows existence and uniqueness for the SCEM, we briefly elaborate on these matters in the following since the coercivity constant associated with the weak formulation of the SCEM relates to the radius of convergence for the presented Taylor series. The motivation behind our choice of norms stems from maximising the radius of convergence.
	
	Let $E_j \subset \partial\Omega$, $j = 1,\dots,m$, be nonempty, connected, relatively open, and such that their closures are mutually disjoint. The surface patch $E_j$ corresponds to the location of the $j$'th electrode in a practical setting. Denote $E = \cup_{j=1}^m E_j$ and assume that the contact admittance $\zeta\in L^\infty(\partial \Omega)$ satisfies
	\begin{equation*} 
		\redel(\zeta) \geq 0 \quad\text{and}\quad \zeta|_{\partial\Omega\setminus \overline{E}} \equiv 0.
	\end{equation*}
	In order to enable current flow through all electrodes, we assume there exist $c_\zeta>0$ and nonempty open subsets $\mathcal{E}_j\subset E_j$, $j=1,\dots,m$, such that
	\begin{equation*}
		\redel(\zeta) \geq c_\zeta \quad \text{a.e.\ on } \mathcal{E} = \cup_{j=1}^m \mathcal{E}_j.
	\end{equation*}
    We define
	\begin{equation*}
		\inner{w,v}_\zeta = \int_{\partial\Omega} \zeta w\overline{v}\,\di S,\quad v,w\in L^2(\partial\Omega),
	\end{equation*}
    in anticipation of it becoming a part of the sesquilinear form for the SCEM. In our analysis, it is assumed that $\{E_j\}_{j=1}^m$ and $\zeta$ are known, that is, the conductivity coefficient is the only unknown in the considered inverse problem.
    
	Define $\mathbf{1} = [1,\dots,1]\in\mathbb{C}^m$ and consider mean free electrode current patterns on the hyperplane
	\begin{equation*}
		\mathbb{C}_\diamond^m = \{ W \in \mathbb{C}^m \mid W\cdot\mathbf{1} = 0 \}.
	\end{equation*}
	%
        The SCEM states that the pair $(u,U)\in H^1(\Omega)\oplus \mathbb{C}^m$, consisting of the electric potential in $\Omega$ and those on the electrodes, satisfies
	\begin{alignat*}{2}
		-\nabla\cdot(A\nabla u) &= 0& &\text{in } \Omega, \\[1mm]
		\nu\cdot(A\nabla u) &= \zeta(U-u)& &\text{on } \partial\Omega, \\[-1mm]
		\int_{E_j}\nu\cdot(A\nabla u)\,\di S &= I_j,& \quad&j = 1,\dots,m,
	\end{alignat*}
        for a conductivity coefficient $A\in L_+^\infty(\Omega)$, an electrode current pattern $I\in\mathbb{C}_\diamond^m$, and a contact admittance $\zeta$ with the above listed properties.
	Here the vector of electrode potentials $U \in \C^m$ is identified with the piecewise constant function
	\begin{equation} \label{eq:pwcfun}
		\sum_{j=1}^m U_j \chi_{E_j} \quad \text{on } \partial \Omega,
	\end{equation}
	where $\chi_{E_j}$ is the characteristic function of $E_j$.
	
	Let $\mathcal{H} = (H^1(\Omega)\oplus \mathbb{C}^m)/\mathbb{C}$ denote the Hilbert space whose elements are equivalence classes of elements in $H^1(\Omega)\oplus \mathbb{C}^m$, with the associated equivalence relation
	\begin{equation*}
		(v,V) \sim (w,W) \quad\Longleftrightarrow\quad \exists c\in\mathbb{C} : \ v-w \equiv c \enskip\wedge\enskip V-W = c\mathbf{1}.
	\end{equation*}
	We equip $\mathcal{H}$ with the norm
	\begin{equation} \label{eq:Hnorm}
		\norm{(v,V)}_{\mathcal{H}}^2 = \norm{v}^2_{*} + \norm{\widetilde{T}v-V}_{L^2(\mathcal{E})}^2,
	\end{equation}
	where $V$ is  once again identified with a piecewise constant function as in \eqref{eq:pwcfun}. Here $\widetilde{T}: H^1(\Omega) \to H^{1/2}(\partial \Omega) \subset L^2(\partial \Omega)$ is the Dirichlet trace operator onto $\partial\Omega$. It should be noted that \eqref{eq:Hnorm} is equivalent to the more standard quotient norm of $\mathcal{H}$ introduced in~\cite{Hyvonen2017,Somersalo1992}:
	\begin{equation*} 
		\norm{(v,V)}_{\widetilde{\mathcal{H}}}^2 = \inf_{c\in\mathbb{C}}\bigl(\norm{v-c}^2_{H^1(\Omega)} + \abs{V-c\mathbf{1}}^2\bigr).
	\end{equation*}
	Indeed, the inequality $\norm{(v,V)}_{\widetilde{\mathcal{H}}}\leq C_{\mathcal{H}}\norm{(v,V)}_{\mathcal{H}}$ is proven in \cite[Proof of Lemma~2.1]{Hyvonen2017}, while the other direction $\norm{(v,V)}_{\mathcal{H}}\leq C_{\widetilde{\mathcal{H}}}\norm{(v,V)}_{\widetilde{\mathcal{H}}}$ follows via a similar line of reasoning as \cite[Lemma~3.2]{Somersalo1992}; such an estimate also holds if $\mathcal{E}$ in \eqref{eq:Hnorm} is replaced by $\partial\Omega$, which is relevant for proving the continuity of the sesquilinear form below. See also \cite[Lemma~2.5]{Hyvonen2004}.
	
	The weak form of the SCEM, as given in \cite{Hyvonen2017}, is
	\begin{equation} \label{eq:weakCEM}
		a_A[(u,U),(v,V)] = I\cdot \overline{V}, \quad \forall (v,V)\in \mathcal{H},
	\end{equation}
	where the sesquilinear form on the left is defined by
	\begin{equation}
		a_A[(w,W),(v,V)] = \inner{w,v}_A + \inner{\widetilde{T}w-W,\widetilde{T}v-V}_\zeta,\quad (w,W),(v,V)\in\mathcal{H}. \label{eq:SCEMform}
	\end{equation} 
	The following continuity and coercivity estimates can be deduced by comparing \eqref{eq:Hnorm} and \eqref{eq:SCEMform}: 
	\begin{align}
		\abs{a_A[(w,W),(v,V)]} &\leq (\norm{A}+C\norm{\zeta}_{L^\infty(\partial\Omega)})\norm{(w,W)}_\mathcal{H}\norm{(v,V)}_\mathcal{H}, \notag\\[1mm]
		\abs{a_A[(v,V),(v,V)]} &\geq \redel(a_A[(v,V),(v,V)]) \geq \min\{c_A,c_\zeta\}\norm{(v,V)}_\mathcal{H}^2. \label{eq:coercive2}
	\end{align}
	On the other hand, for any $I\in\mathbb{C}_\diamond^m$ we have 
	\begin{equation} \label{eq:LM_SCEM_RHS}
		\abs{I\cdot \overline{V}} = \inf_{c\in\mathbb{C}}\abs{I\cdot \overline{(V-c\mathbf{1})}} \leq \abs{I}\inf_{c\in\mathbb{C}}\abs{V-c\mathbf{1}} \leq C_{\mathcal{H}}\abs{I}\norm{(v,V)}_{\mathcal{H}}, \quad (v,V)\in\mathcal{H}.
	\end{equation}
	Hence, the Lax--Milgram lemma guarantees \eqref{eq:weakCEM} has a unique solution $(u,U)\in\mathcal{H}$ that satisfies the bound
	\begin{equation} \label{eq:ubound2}
		\norm{(u,U)}_{\mathcal{H}} \leq \frac{C_{\mathcal{H}}\abs{I}}{\min\{c_A,c_\zeta\}}.
	\end{equation}
	We occasionally write $(u, U) = (u_I^A,U_I^A)$ for the unique solution of \eqref{eq:weakCEM} to be more specific about its connection to $I$ and $A$. 
	
	\begin{remark}
		If $z_j\in\mathbb{C}$ with $\redel(z_j)>0$, $j=1,\dots,m$, are constant contact impedances on the electrodes, then the standard CEM is obtained from the SCEM by setting $\zeta|_{E_j} \equiv z_j^{-1}$. In this case, $c_\zeta = \min_j [\redel(z_j) / \abs{z_j}^2]$ and $\mathcal{E} = E$.
	\end{remark}
	
	We denote by $N_{\textup{E}} : L^\infty_+(\Omega) \to \mathscr{L}(\mathbb{C}_\diamond^m,\mathcal{H})$ the map that sends a conductivity coefficient to the corresponding solution operator of the SCEM, i.e.\ $N_{\textup{E}}(A)I = (u_I^A,U_I^A)$. Moreover, we use $T_1 \in\mathscr{L}(\mathcal{H}, H^1(\Omega))$ and $T_2\in\mathscr{L}(\mathcal{H},\mathbb{C}_\diamond^m)$ to extract components from $(v,V)\in\mathcal{H}$. To be more precise, $T_1(v,V) = \tilde{v}$ and $T_2(v,V) = \widetilde{V}$, where $(\tilde{v},\widetilde{V})\in H^1(\Omega)\oplus \mathbb{C}^m$ is the unique element in the equivalence class $(v,V)\in\mathcal{H}$ with a mean free second component $\widetilde{V}\in\mathbb{C}_\diamond^m$. Note that forcing the electrode potential to be mean free corresponds to a systematic way of choosing the ground level of potential.
	
	Let $(v,V)\in\mathcal{H}$ be arbitrary. Due to orthogonality, $\abs{\widetilde{V}-c\mathbf{1}}^2 = \abs{\widetilde{V}}^2+d\abs{c}^2\geq d\abs{c}^2$ for $\widetilde{V} = T_2(v,V)$.  Hence, $\tilde{v} = T_1(v,V)$ satisfies
	\begin{align*}
		\norm{\tilde{v}}_{H^1(\Omega)} &= \inf_{c\in\mathbb{C}}\norm{\tilde{v}-c+c}_{H^1(\Omega)} \leq \sqrt{2}\inf_{c\in\mathbb{C}}\left( \norm{\tilde{v}-c}_{H^1(\Omega)}^2+\abs{c}^2\abs{\Omega} \right)^{1/2} \\
		&\leq C_\Omega\norm{(\tilde{v},\widetilde{V})}_{\widetilde{\mathcal{H}}} = C_\Omega\norm{(v,V)}_{\widetilde{\mathcal{H}}} \leq C_\mathcal{H}C_\Omega \norm{(v,V)}_{\mathcal{H}}, 
	\end{align*}
	with $C_\Omega = \sqrt{2\max\{1,\abs{\Omega}/d\}}$, where $\abs{\Omega}$ is the domain's Lebesgue measure. Considering also \eqref{eq:LM_SCEM_RHS} with $I = \widetilde{V}$, we have altogether established that
	\begin{equation}
		\norm{T_1}_{\mathscr{L}(\mathcal{H},H^1(\Omega))} \leq C_\mathcal{H}C_\Omega, \qquad \norm{T_2}_{\mathscr{L}(\mathcal{H},\mathbb{C}_\diamond^m)} \leq C_\mathcal{H}, \label{eq:normT}
	\end{equation} 
	where $\mathbb{C}_\diamond^m$ is equipped with the Euclidean norm.
	
	As for the CM, we introduce a few  mappings related to the SCEM:
	\begin{enumerate}[(i)]
		\item $F_{\textup{E}} : L^\infty_+(\Omega) \to \mathscr{L}(\mathbb{C}_\diamond^m,H^1(\Omega))$ defined by $F_{\textup{E}}(A)I = T_1 N_{\textup{E}}(A)I = \tilde{u}_I^A$.
		\item $\Lambda_{\textup{E}} : L^\infty_+(\Omega) \to \mathscr{L}(\mathbb{C}_\diamond^m)$ defined by $\Lambda_{\textup{E}}(A)I = T_2 N_{\textup{E}}(A)I = \widetilde{U}_I^A$.
		\item $P_{{\textup{E}},A} \in \mathscr{L}(L^\infty(\Omega;\mathbb{C}^{d\times d}),\mathscr{L}(\mathcal{H}))$ for $A\in L^\infty_+(\Omega)$ defined by
		\begin{equation} \label{eq:SCEM_P}
			a_A[P_{{\textup{E}},A}(B)(y,Y),(v,V)] = -\inner{y,v}_B, \quad \forall (v,V)\in \mathcal{H},
		\end{equation}
		where $B \in L^\infty(\Omega, \C^{d \times d})$ and $(y,Y) \in \mathcal{H}$.
	\end{enumerate}
	As in the case of the CM, $P_{{\textup{E}},A}$ is well-defined due to the Lax--Milgram lemma that guarantees the unique solvability of the variational problem \eqref{eq:SCEM_P} and also yields the estimate
	\begin{equation}
		\norm{P_{{\textup{E},A}}(B)}_{\mathscr{L}(\mathcal{H})} \leq \frac{\norm{B}}{\min\{c_A,c_\zeta\}}. \label{eq:PEbound}
	\end{equation}
	The nonlinear map $\Lambda_{\textup{E}}$ is called the forward map of the SCEM; it maps a given conductivity coefficient $A$ to the corresponding electrode current-to-voltage operator $\Lambda_{\textup{E}}(A)$.
	
	\section{Taylor series for the SCEM}
        \label{sec:taylor_SCEM}
	
	We obtain analogous differentiability and analyticity results for the SCEM as for the CM in Section~\ref{sec:taylor}. The fundamental ideas of the proofs are also the same, with the main difference being the employment of a different coercive sesquilinear form.
	\begin{lemma}
		\label{lemma:perturbop2}
		$P_{\textup{E},A}(B)$, and more generally $P_{\textup{E},A}$, is infinitely times continuously Fr\'echet differentiable with respect to $A$. Its first derivative $D_A P_{\textup{E},A}(B)$ is given by
		\begin{equation*}
			D_A P_{\textup{E},A}(B;C) = P_{\textup{E},A}(C)P_{\textup{E},A}(B)
		\end{equation*}
		for $A \in L^\infty_+(\Omega)$ and $B, C \in L^\infty(\Omega; \mathbb{C}^{d\times d})$.
	\end{lemma}
	\begin{proof}
		Let $\norm{C}$ be small enough so that $A+C\in L^\infty_+(\Omega)$. For a given $(y,Y)\in\mathcal{H}$, we let $(w_{A+C},W_{A+C}) = P_{\textup{E},A+C}(B)(y,Y)$ and $(w_{A},W_{A}) = P_{\textup{E},A}(B)(y,Y)$. According to the definitions of $P_{\textup{E},A+C}(B)$ and $P_{\textup{E},A}(B)$, we have
		\begin{equation*}
			a_{A+C}[(w_{A+C},W_{A+C}),(v,V)] = -\inner{y,v}_B = a_{A}[(w_{A},W_{A}),(v,V)],\quad \forall (v,V)\in \mathcal{H}.
		\end{equation*}
		Restructuring gives
		\begin{equation*}
			a_{A}[(w_{A+C}-w_A,W_{A+C}-W_A),(v,V)] = -\inner{w_{A+C},v}_C,
		\end{equation*}
		for all $(y,Y),(v,V)\in \mathcal{H}$ and $B\in L^\infty(\Omega; \mathbb{C}^{d\times d})$. Using the definition of $P_{\textup{E},A}(C)$, we deduce that $(w_{A+C}-w_A,W_{A+C}-W_A)=P_{\textup{E},A}(C)(w_{A+C},W_{A+C})$,~i.e.
		\begin{equation}
			P_{\textup{E},A+C}(B) - P_{\textup{E},A}(B) = P_{\textup{E},A}(C)P_{\textup{E},A+C}(B). \label{eq:Pequality2}
		\end{equation}
		The proof can now be completed by mimicking that of Lemma~\ref{lemma:perturbop}, with \eqref{eq:Pequality2} and \eqref{eq:PEbound} taking the roles of \eqref{eq:Pequality} and \eqref{eq:Pbound}, respectively.
	\end{proof}
	
	This leads to the following result on Taylor series representations.
	
	\begin{theorem} \label{thm:NDdiff2}
        The mappings $N_{\textup{E}}$, $F_{\textup{E}}$, and $\Lambda_{\textup{E}}$ are infinitely times continuously Fr\'echet differentiable. Their derivatives at $A \in L^\infty_+(\Omega)$ are given by $D^k \! F_{\textup{E}}(A) = T_1D^k \! N_{\textup{E}}(A)$, $D^k \! \Lambda_{\textup{E}}(A) = T_2D^k \! N_{\textup{E}}(A)$, and
		\begin{equation}
			D^k \! N_{\textup{E}}(A; B_1, \ldots, B_k) = \Bigl[ \sum_{\alpha \in \rho_k} P_{\textup{E},A}(B_{\alpha_1})\cdots P_{\textup{E},A}(B_{\alpha_k}) \Bigr]N_\textup{E}(A) \label{eq:Nderiv2}
		\end{equation}
		for $k\in\mathbb{N}$ and $B_1, \dots, B_k \in L^\infty(\Omega; \mathbb{C}^{d\times d})$. These mappings are analytic with Taylor series
		\begin{align}
			N_{\textup{E}}(A + B) &= \sum_{k = 0}^\infty \frac{1}{k!} D^k\! N_{\textup{E}}(A; B, \ldots, B) = \sum_{k=0}^\infty P_{\textup{E},A}(B)^k N_{\textup{E}}(A), \label{eq:Nseries2} \\
			F_{\textup{E}}(A + B) &= \sum_{k = 0}^\infty \frac{1}{k!} D^k\! F_{\textup{E}}(A; B, \ldots, B) = T_1\sum_{k=0}^\infty P_{\textup{E},A}(B)^k N_{\textup{E}}(A), \notag \\
			\Lambda_{\textup{E}}(A + B) &= \sum_{k = 0}^\infty \frac{1}{k!} D^k\! \Lambda_{\textup{E}}(A; B, \ldots, B) = T_2\sum_{k=0}^\infty P_{\textup{E},A}(B)^k N_{\textup{E}}(A) \notag
		\end{align}
		for $A\in L^\infty_+(\Omega)$ and $B \in L^\infty(\Omega; \mathbb{C}^{d\times d})$ such that $\norm{B} < \min\{c_A,c_\zeta\}$. 
	\end{theorem}
	\begin{proof}
		It is sufficient to prove the expansion \eqref{eq:Nseries2}. The specific formulas for the derivatives can then be derived from the product formula and Lemma~\ref{lemma:perturbop2}. 
		
		Let $A\in L^\infty_+(\Omega)$ and $B\in L^\infty(\Omega; \mathbb{C}^{d\times d})$, with $\norm{B}$ small enough so that also $A+B\in L^\infty_+(\Omega)$. For a given $I \in \C_\diamond^m$, we let $(w_{A+B},W_{A+B}) = N_{\textup{E}}(A+B)I$ and $(w_A,W_A) = N_{\textup{E}}(A)I$. The definitions of $N_{\textup{E}}(A+B)$ and $N_{\textup{E}}(A)$ yield
		\begin{equation*}
			a_{A+B}[(w_{A+B},W_{A+B}),(v,V)] = I\cdot \overline{V} = a_{A}[(w_{A},W_{A}),(v,V)],
		\end{equation*}
		or equivalently,
		\begin{equation}
			-\inner{w_{A+B},v}_{B} = a_{A}[(w_{A+B}-w_A,W_{A+B}-W_A),(v,V)] \label{eq:Ntmp2}
		\end{equation}
		for all $I\in\mathbb{C}_\diamond^m$ and $(v,V)\in \mathcal{H}$. Combining the definition of $(w,W) = P_{\textup{E},A}(B)(w_{A+B},W_{A+B})$ with \eqref{eq:Ntmp2} gives
		\begin{equation*}
			a_{A}[(w,W),(v,V)] = -\inner{w_{A+B},v}_{B} = a_{A}[(w_{A+B}-w_A,W_{A+B}-W_A),(v,V)]
		\end{equation*}
		for all $I\in\mathbb{C}_\diamond^m$ and $(v,V)\in \mathcal{H}$. Due to \eqref{eq:coercive2}, we have actually shown that $(w,W) = (w_{A+B}-w_A,W_{A+B}-W_A)$ for all $I \in \C_\diamond^m$, i.e.\
		\begin{equation*}
			P_{\textup{E},A}(B)N_{\textup{E}}(A+B) = N_\textup{E}(A+B)-N_\textup{E}(A),
		\end{equation*}
		which may be rewritten as
		\begin{equation*}
			(\ident-P_{\textup{E},A}(B))N_\textup{E}(A+B) = N_\textup{E}(A).
		\end{equation*}

        Assume that $\norm{B} < \min\{c_A,c_\zeta\}$, which guarantees $A+B\in L^\infty(\Omega; \mathbb{C}^{d\times d})$ and $\norm{P_{\textup{E},A}(B)}_{\mathscr{L}(\mathcal{H})} < 1$ by virtue of \eqref{eq:PEbound}. The latter allows inverting $\ident-P_{\textup{E},A}(B)$ via a Neumann series, which gives the sought-for expansion in \eqref{eq:Nseries2}. 
	\end{proof}
	As in Section~\ref{sec:taylor}, we can easily derive a more standard formula for the first Fr\'echet derivative of $\Lambda_{\textup{E}}$ if $A\in L^\infty_+(\Omega)$ is Hermitian and $\zeta$ is real-valued, making $a_A$ a symmetric sesquilinear form. Obviously, $D\!\Lambda_{\textup{E}}(A;B)J = T_2P_{\textup{E},A}(B)(u_J^A,U_J^A)$, which by \eqref{eq:weakCEM} and the definition of $P_{\textup{E},A}(B)$ gives
	\begin{equation}
		(D\!\Lambda_\textup{E}(A;B)J)\cdot \bar{I} = a_A[P_{\textup{E},A}(B)(u_J^A,U_J^A),(u_I^A,U_I^A)] = -\inner{u_J^A,u_I^A}_B, \quad I,J\in \mathbb{C}_\diamond^m.
	\end{equation}

        The following bounds are immediate consequences of \eqref{eq:ubound2}--\eqref{eq:PEbound} and \eqref{eq:Nderiv2}.
	\begin{corollary} \label{col:scem}
	Let $A\in L^\infty_+(\Omega)$ and $k\in \mathbb{N}$. Then
	\begin{align*}
		\norm{ D^k\! N_\textup{E}(A) }_{\mathscr{L}^k(L^\infty(\Omega; \mathbb{C}^{d\times d}), \mathscr{L}(\mathbb{C}_\diamond^m,\mathcal{H}))} &\leq \frac{k!C_{\mathcal{H}}}{\min\{c_A,c_\zeta\}^{k+1}}, \\
		\norm{ D^k\! F_\textup{E}(A) }_{\mathscr{L}^k(L^\infty(\Omega; \mathbb{C}^{d\times d}), \mathscr{L}(\mathbb{C}_\diamond^m,H^1(\Omega)))} &\leq \frac{k!{C_{\mathcal{H}}}^2 C_\Omega}{\min\{c_A,c_\zeta\}^{k+1}}, \\
		\norm{ D^k\! \Lambda_\textup{E}(A) }_{\mathscr{L}^k(L^\infty(\Omega; \mathbb{C}^{d\times d}), \mathscr{L}(\mathbb{C}_\diamond^m))} &\leq \frac{k!{C_{\mathcal{H}}}^2}{\min\{c_A,c_\zeta\}^{k+1}}.
	\end{align*}
	\end{corollary}

	\begin{remark}
	  As for the CM, one could slightly vary the domains and codomains of the operators appearing in Theorem~\ref{thm:NDdiff2}; all such changes would correspond to fixing the ground level of potential in different ways or alternatively leaving it unspecified. As a concrete example, the codomain of $\Lambda_{\textup{E}}(A)$ could as well be $\mathbb{C}^m/\mathbb{C}$.
	\end{remark}

    \begin{remark}
      One could also consider the differentiability of the SCEM with respect to the contact admittance, which would lead to similar analysis. See \cite[Theorem~2.6]{Darde21} for a formula of the first derivative with respect to $\zeta$ in a simpler framework involving only scalar-valued conductivities. 
    \end{remark}
	
	\section{Series reversion} \label{sec:reversion}
	
	Let us fix the known coefficient $A\in L^\infty_+(\Omega)$ and consider an additive perturbation $B\in L^\infty(\Omega; \mathbb{C}^{d\times d})$. Our aim is to deduce asymptotic formulas in powers of $\norm{B}$ for reconstructing $B$ based on partial knowledge of the measurement $\Lambda(A+B)$ or $\Lambda_{\rm E}(A+B)$ combined with {\em a priori} information on a subset of $L^\infty(\Omega; \mathbb{C}^{d\times d})$ containing $B$. Throughout this section it is assumed that $A+B\in L^\infty_+(\Omega)$. 
	
	As preparation for the series reversion results, let us introduce some auxiliary concepts:
	\begin{itemize}
		\item CM: Let $\mathcal{Z} = L^2_\diamond(\Gamma)$. SCEM: Let $\mathcal{Z} = \mathbb{C}_\diamond^m$.
		\item Let $\mathcal{X} \subseteq \mathcal{Z}$ be a closed subspace and $\mathscr{P} \in \mathscr{L}(\mathcal{Z},\mathcal{X})$ be the orthogonal projection onto $\mathcal{X}$.
		\item CM: Let $\mathscr{F} = \mathscr{P}D\!\Lambda(A;\,\cdot\,)\mathscr{P}$. SCEM: Let $\mathscr{F} = \mathscr{P}D\!\Lambda_\textup{E}(A;\,\cdot\,)\mathscr{P}$.
		\item Let $\mathcal{W}\subset L^\infty(\Omega;\mathbb{C}^{d\times d})$ be a closed subspace and define $\mathcal{Y} = \mathscr{F}(\mathcal{W})$. 
	\end{itemize}
	The following, arguably rather restrictive, assumption is needed for our analysis.
	\begin{assumption} \label{assump}
		Assume that $B\in \mathcal{W}$, $\mathscr{F}$ is injective on $\mathcal{W}$, and $\mathcal{Y}$ is closed and complemented in $\mathscr{L}_\textup{C}(\mathcal{X})$.
	\end{assumption}
    Due to Assumption~\ref{assump} and the inverse mapping theorem, $\mathscr{F}\in\mathscr{L}(\mathcal{W},\mathcal{Y})$ has a bounded inverse $\mathscr{F}^{-1} \in \mathscr{L}(\mathcal{Y},\mathcal{W})$. Moreover, it is evident that $\mathcal{W}$ must be finite-dimensional for the SCEM as $\mathcal{Y} \subseteq \mathscr{L}(\mathbb{C}_\diamond^m)$, and the same holds for the CM if $\mathcal{X}$ is finite-dimensional. The last condition of Assumption~\ref{assump} also allows us to introduce one more tool to be used in the following:
	\begin{itemize}
	    \item Let $Q \in \mathscr{L}(\mathscr{L}_\textup{C}(\mathcal{X}),\mathcal{Y})$ be a projection onto $\mathcal{Y}$.
	\end{itemize}
 	Appendix~\ref{sec:HS} introduces a systematic way of forming $Q$ when $d=2$ and $\Gamma = \partial \Omega$. Moreover, as detailed in the remark below, the last condition in Assumption~\ref{assump} automatically holds if $\mathcal{W}$ is finite-dimensional. In Section~\ref{sec:implementation} we also show how to avoid forming $Q$ in practice if $\mathcal{W}$ is finite-dimensional. To our knowledge, it is not known if the last condition in Assumption~\ref{assump} can be satisfied by an infinite-dimensional space $\mathcal{W}$. Nevertheless we leave the method open to such a possibility.
        
	\begin{remark}
	  The existence of a linear and bounded projection $Q$ is based on $\mathcal{Y}$ being a complemented subspace. This is always true if $\mathcal{W}$ is finite-dimensional. Because $T$ is compact, since its codomain is considered as $L^2$, so are all operators of $\mathscr{L}(L^2_\diamond(\Gamma))$ considered in Sections~\ref{sec:continuum} and \ref{sec:taylor}. Hence, we can assume $\mathcal{Y}$ is complemented in $\mathscr{L}_\textup{C}(\mathcal{X})$ rather than in the larger space $\mathscr{L}(\mathcal{X})$. Notice that there are known issues with existence of linear bounded projections from,~e.g.,\ the space of bounded linear operators to the space of compact operators, if the latter is a proper subspace \cite{Thorp1960}.
	\end{remark}
	
	The assumed datum for the CM is $\mathscr{P}\Lambda(A+B)\mathscr{P}$, and that for the SCEM is $\mathscr{P}\Lambda_\textup{E}(A+B)\mathscr{P}$. The projection $\mathscr{P}$ enables the use of a finite number of measurements in the framework of the CM, but we also allow $\mathscr{P}$ to be the identity, which corresponds to the standard infinite-dimensional datum. For the SCEM, $\mathscr{P}$ enables using fewer than the maximal number of linearly independent current patterns.
	
    For the CM, we introduce $M \in \mathscr{L}(\mathscr{L}(L^2_\diamond(\Gamma),H_\diamond^1(\Omega)),\mathcal{W})$ via
    \begin{equation*}
    	M(y) = \mathscr{F}^{-1}Q(\mathscr{P}Ty\mathscr{P})
    \end{equation*}
    and define
	\begin{equation*}
		\quad P = P_A, \quad N = N(A), \quad F_1 = \mathscr{F}^{-1}Q(\mathscr{P}\Lambda(A+B)\mathscr{P}-\mathscr{P}\Lambda(A)\mathscr{P}). 
	\end{equation*}
	For the SCEM, the corresponding definitions are $M \in \mathscr{L}(\mathscr{L}(\C_\diamond^m,\mathcal{H}),\mathcal{W})$ such that
    \begin{equation*}
        M(y) = \mathscr{F}^{-1}Q(\mathscr{P}T_2y\mathscr{P})
    \end{equation*}
    and
	\begin{equation*}
		P = P_{\textup{E},A}, \quad N = N_\textup{E}(A), \quad F_1 = \mathscr{F}^{-1}Q(\mathscr{P}\Lambda_\textup{E}(A+B)\mathscr{P}-\mathscr{P}\Lambda_\textup{E}(A)\mathscr{P}).
	\end{equation*}
	 
	Note that $F_1 \in \mathcal{W}$ is a certain solution to the linearised problem. In what follows, we avoid using parentheses in connection with $M$; at each occurrence, $M$ operates on the whole composition of operators to its right.
	
	Since the CM and SCEM have Taylor series of the same overall structure, we also obtain series reversions of the same type. Focusing on the CM to begin with, Taylor's theorem entails
	\begin{align}
		D\!\Lambda(A;B) &= \Lambda(A+B) - \Lambda(A) - T\sum_{k=2}^K P(B)^k N + O(\norm{B}^{K+1}), \notag \\
		\mathscr{F}(B) = Q(\mathscr{F}(B)) &= Q \Big( \mathscr{P}(\Lambda(A+B) - \Lambda(A))\mathscr{P} - \mathscr{P}T\sum_{k=2}^K P(B)^k N\mathscr{P} + \mathscr{P}O(\norm{B}^{K+1})\mathscr{P}\Big) \notag\\[-1mm]
		 &= \mathscr{F}(F_1) - Q\Big(\mathscr{P}T\sum_{k=2}^K P(B)^k N\mathscr{P} \Big) + Q\big(\mathscr{P}O(\norm{B}^{K+1})\mathscr{P}\big), \notag\\
		B &= F_1 - M\sum_{k=2}^K P(B)^k N + O(\norm{B}^{K+1}). \label{eq:recursion}
	\end{align}
	Exactly the same line of reasoning also leads to the representation \eqref{eq:recursion} for the SCEM, but with the corresponding ``SCEM definitions'' for the involved operators and functions.

    Notice that the remainder term in \eqref{eq:recursion} is $O(\norm{B}^{K+1})$ instead of the slightly weaker version $o(\norm{B}^{K})$. This follows from the integral form for the remainder term in Taylor's theorem since both $D^{K+1}\!\Lambda$ and $D^{K+1}\!\Lambda_{\textup{E}}$ are bounded by Corollaries~\ref{col:cm} and \ref{col:scem}, respectively, in any $A$-centered ball of radius less than $c_A$ in the topology for $L^\infty(\Omega;\mathbb{C}^{d\times d})$. In particular, they are bounded on the line segment
	\begin{equation*}
		[A,A+B] = \{ A + tB \mid t\in [0,1] \} \subset L_+^\infty(\Omega)
	\end{equation*}
	connecting $A$ and $A+B$ for small enough $B$.
	
	The following theorem, which presents our main result on the series reversion for Calder{\'o}n's problem, makes use of abbreviations for certain elements of $\mathcal{W}$:
	\begin{align*}
		G_k &= MP(F_1)^kN, \\
		G_{2,2} &= MP(G_2)^2N, \\
		L_{k,1} &= MP(F_1)P(G_k)N, \\
		R_{k,1} &= MP(G_k)P(F_1)N, \\
		LL_{2,2} &= MP(F_1)P(L_{2,1})N, \\
		RR_{2,2} &= MP(R_{2,1})P(F_1)N, \\
		RL_{2,2} &= MP(L_{2,1})P(F_1)N, \\
		LR_{2,2} &= MP(F_1)P(R_{2,1})N, \\
		C_{2,2} &= MP(F_1)P(G_2)P(F_1)N, \\
		V_{2,2} &= MP(F_1)^2P(G_2)N, \\
		H_{2,2} &= MP(G_2)P(F_1)^2N.
	\end{align*}
	The notation with $L$ and $R$ indicates the relative position (left/right) of the operator $P(F_1)$. The same applies to $V_{2,2}$ and $H_{2,2}$, with the letters originating from the Danish words ``venstre''/``h{\o}jre'' for ``left''/``right''. The sum of the indices gives the order of each term relative to powers of~$\norm{B}$.
	
	\begin{theorem} \label{thm:main}
		Let Assumption~\ref{assump} be satisfied. Then there exist matrix-valued functions $F_j\in L^\infty(\Omega;\mathbb{C}^{d\times d})$, $j \in \N$, only depending on $M$, $P$, $N$ and $F_1$, such that $\norm{F_j} = O(\norm{B}^j)$ and
		\begin{equation} \label{eq:Bseries}
			B = \sum_{j=1}^K F_j + O(\norm{B}^{K+1})
		\end{equation}
		for any $K\in\mathbb{N}$. The first term $F_1$ is given above and the following three are
		%
		\begin{align}
        	F_2 &= -G_2, \label{eq:F2} \\
            F_3 &= -G_3 + L_{2,1} + R_{2,1}, \label{eq:F3} \\
            F_4 &= - G_4 - G_{2,2} + C_{2,2} + V_{2,2} + L_{3,1} - LL_{2,2} - LR_{2,2} + H_{2,2} + R_{3,1} - RR_{2,2} - RL_{2,2}. \label{eq:F4}
		\end{align}
	\end{theorem}
	
	\begin{proof}
	As hinted in the statement of the theorem, we will only consider the case $K = 4$. Higher order approximations can be found analogously, but their deduction requires a sufficient level of tenacity as the number of individual terms in $F_j$ increases rapidly with $j$. Many terms that are included in the remainder term in the presented proof would also have to be taken into account.
	
	We derive a series reversion of order 4 by recursively inserting \eqref{eq:recursion} into its right-hand side and collecting all terms of order $\norm{B}^5$ and higher in the remainder term. We first rewrite \eqref{eq:recursion} as
	\begin{equation}
		B = F_1 - MP(B)^2N  - MP(B)^3N - MP(B)^4N + O(\norm{B}^5). \label{eq:recursion4}
	\end{equation}
	Let us then focus on $P(B)^2$:
	\begin{align}
		P(B)^2 &= P\!\left( F_1 - MP(B)^2N - MP(B)^3N - MP(B)^4N\right)^2 + O(\norm{B}^5) \notag\\
		&= P(F_1)^2 + P(MP(B)^2N)^2 - P(F_1)P(MP(B)^2N) - P(MP(B)^2N)P(F_1)   \notag\\
		&\hphantom{{}={}}- P(F_1)P(MP(B)^3N) - P(MP(B)^3N)P(F_1) + O(\norm{B}^5). \label{eq:PB2}
	\end{align}
	Abbreviating $\eta = MP(B)^2N$, and initially focusing on terms of at most order $\norm{B}^2$, we deduce from \eqref{eq:PB2} that
	\begin{equation} \label{eq:Petalow}
		P(\eta) =  P(G_2) + O(\norm{B}^3).
	\end{equation}
	Hence, by \eqref{eq:Petalow} the only term in $P(\eta)^2$ of at most order $\norm{B}^4$ is $P(G_2)^2$, 
	\begin{equation}
		P(\eta)^2 = P(G_2)^2 + O(\norm{B}^5). \label{eq:Petasquared}
	\end{equation}
	To handle $P(F_1)P(\eta)$ and $P(\eta)P(F_1)$ in \eqref{eq:PB2}, we need to determine the terms of at most order $\norm{B}^3$ in $P(\eta)$, namely
	\begin{align}
		P(\eta) &= P(G_2) - P(MP(F_1)P(\eta)N) - P(MP(\eta)P(F_1)N) + O(\norm{B}^4) \notag \\
		&= P(G_2) - P(L_{2,1}) - P(R_{2,1}) + O(\norm{B}^4) \label{eq:Peta},
	\end{align}
	where we used \eqref{eq:Petalow}.
	 
	Inserting \eqref{eq:Petasquared} and \eqref{eq:Peta} into \eqref{eq:PB2}, and collecting the higher order terms in the remainder, leads to
	\begin{align}
		P(B)^2 &=  P(F_1)^2 + P(G_2)^2 - P(F_1)[P(G_2) + P(G_3) - P(L_{2,1}) - P(R_{2,1})]  \notag \\
		&\hphantom{{}={}}- [P(G_2) + P(G_3) - P(L_{2,1}) - P(R_{2,1})]P(F_1) + O(\norm{B}^5), \label{eq:PB22}
	\end{align}
	or rather,
	\begin{align*}
		MP(B)^2N &=  G_2 + G_{2,2} - L_{2,1} - L_{3,1} + LL_{2,2} + LR_{2,2} \notag\\
		&\hphantom{{}={}}- R_{2,1} - R_{3,1} + RL_{2,2} + RR_{2,2}  + O(\norm{B}^5). 
	\end{align*}
	Since $MP(B)^4N = G_4 + O(\norm{B}^5)$, the only remaining term to be investigated in \eqref{eq:recursion4} is $MP(B)^3N$. Making use of \eqref{eq:recursion4}, \eqref{eq:Peta}, and \eqref{eq:PB22}, we obtain
	\begin{align*}
		P(B)^3 &= \bigl[P(F_1) - P(\eta)\bigr]P(B)^2 + O(\norm{B}^5) = \bigl[P(F_1)-P(G_2)\bigr]P(B)^2 + O(\norm{B}^5) \\
		&= P(F_1)^3 - P(G_2)P(F_1)^2 - P(F_1)^2P(G_2) - P(F_1)P(G_2)P(F_1) + O(\norm{B}^5),
	\end{align*}
	or rather,
	\begin{equation*}
		MP(B)^3N = G_3 - H_{2,2} - V_{2,2} - C_{2,2} + O(\norm{B}^5).
	\end{equation*}

        By inserting the above expansions into \eqref{eq:recursion4} and collecting the $j$'th order terms in $F_j$, we finally deduce
	\begin{align*}
	  F_2 &= -G_2, \\[1mm]
	   F_3 &= -G_3 + L_{2,1} + R_{2,1}, \\[1mm]
		F_4 &= - G_4 - G_{2,2} + C_{2,2} + V_{2,2} + L_{3,1} - LL_{2,2} - LR_{2,2} + H_{2,2} + R_{3,1} - RR_{2,2} - RL_{2,2}.
	\end{align*}
	To summarise, we may write
	\begin{equation*}
		B = F_1 + F_2 + F_3 + F_4 + O(\norm{B}^5),
	\end{equation*}
	where $\norm{F_j} = O(\norm{B}^j)$, $j=1, \dots, 4$. 
	\end{proof}

    \begin{remark} \label{rm:structure}
		We observe that the formulas in \eqref{eq:F2}--\eqref{eq:F4} exhibit the following general structure, starting with $\widetilde{P}_1 = 0$ and $F_1$ as above,
		\begin{align*}
			\widetilde{P}_{j} &= \sum_{n=1}^{j-1} P(F_{j-n})(\widetilde{P}_{n} - P(F_n)), \\
			F_{j} &= M\widetilde{P}_{j}N
		\end{align*}
        for $j = 2, 3, 4$. We hypothesise that this recursive scheme actually holds in general,~i.e.~ for any $j \in \mathbb{N}$. If this conjecture is valid, it provides a straightforward method for deriving new higher order series reversion formulas for Calder{\'o}n's problem.
	\end{remark}

    It turns out that under Assumption~\ref{assump} the series \eqref{eq:Bseries} converges for small enough $B\in\mathcal{W}$ as $K$ tends to infinity. This follows from the ``projected relative forward map'' $\mathscr{E}: \mathcal{W}\cap \mathcal{B}_A \to \mathcal{Y}$,
    \begin{equation} \label{eq:forward_ope}
	    \mathscr{E}:
        \Biggl\{
        \begin{array}{ll}
          B \mapsto Q(\mathscr{P}\Lambda(A+B)\mathscr{P}-\mathscr{P}\Lambda(A)\mathscr{P}) & \qquad \text{for the CM}, \\[2mm]
          B \mapsto Q(\mathscr{P}\Lambda_\textup{E}(A+B)\mathscr{P}-\mathscr{P}\Lambda_\textup{E}(A)\mathscr{P}) & \qquad \text{for the SCEM},
        \end{array}
    \end{equation}
    being an analytic diffeomorphism between certain neighbourhoods of the origin in $\mathcal{W}$ and in $\mathcal{Y}$, respectively. Here $\mathcal{B}_A$ is the origin-centered open ball in $L^\infty(\Omega;\mathbb{C}^{d\times d})$ with radius $c_A$, ensuring that $\mathscr{E}$ is well-defined. See also Remark~\ref{remark:conv} below for the connection to the series reversion.
    \begin{theorem}
	    Let Assumption~\ref{assump} be satisfied. Then there exist open neighbourhoods of the origin $U_{\mathcal{W}} \subset \mathcal{W}$ and $U_{\mathcal{Y}} \subset \mathcal{Y}$ such that $\mathscr{E}(U_{\mathcal{W}}) = U_{\mathcal{Y}}$, the restriction $\mathscr{E}|_{U_{\mathcal{W}}}:U_{\mathcal{W}} \to U_{\mathcal{Y}}$ is injective, and it has an analytic inverse $\mathscr{E}^{-1}: U_{\mathcal{Y}} \to U_\mathcal{W}$.
    \end{theorem}
    \begin{proof}
	    Recall that $Q$ is a bounded linear projection onto $\mathcal{Y} = \mathscr{F}(\mathcal{W})$, and recall the definitions $\mathscr{F} =  \mathscr{P} D\!\Lambda(A;\,\cdot\,) \mathscr{P}$ for the CM and $\mathscr{F} = \mathscr{P}D\!\Lambda_\textup{E}(A;\,\cdot\,)\mathscr{P}$ for the SCEM. 
	    
	    By restricting $\mathscr{F}$ to $\mathcal{W}$, it follows immediately from \eqref{eq:forward_ope} that $\mathscr{F} = Q\mathscr{F}$ is the Fr\'echet derivative of $\mathscr{E}$ at the origin,~i.e.~$D\!\mathscr{E}(0; \, \cdot \,) = \mathscr{F}$. As $\mathscr{E}$ is analytic and $\mathscr{F}: \mathcal{W} \to \mathcal{Y}$ is a linear homeomorphism by virtue of Assumption~\ref{assump}, the assertion is a direct consequence of the inverse function theorem for analytic maps between Banach spaces, cf.~\cite[Appendix~I]{Valent1988} and the references therein.
	\end{proof}
    \begin{remark} \label{remark:conv}
    	From the uniqueness of the Taylor series representation for an analytic operator, the first $K$ terms in \eqref{eq:Bseries} indeed correspond to the first $K$ terms of the Maclaurin series for $\mathscr{E}^{-1}$ evaluated at $\mathscr{E}(B)$. In particular, observe that $F_1 = \mathscr{F}^{-1} \mathscr{E}(B)$, which leads to each term $F_j$ in \eqref{eq:Bseries} being a $j$-linear form evaluated at $(\mathscr{E}(B), \dots, \mathscr{E}(B))$.  Moreover, since $\mathscr{F}: \mathcal{W} \to \mathcal{Y}$ is assumed to be a linear homeomorphism, it is easy to see that one could as well write $F_j = O(\norm{\mathscr{E}(B)}_{\mathscr{L}(\mathcal{X})}^j)$.
    \end{remark}
        
    \section{Remarks on numerical implementation}
    \label{sec:implementation}
	
	For simplicity we focus on the isotropic case in our numerical considerations, meaning that $A(x) = a(x)I$ and $B(x) = b(x)I$ where $I$ is an identity matrix, and $a$ and $b$ are bounded measurable scalar-valued functions. Moreover, we assume $a$ is real-valued and has a positive infimum. Naturally, the discretisation of $D\!\Lambda(A)$ would become more involved for anisotropic coefficients as each matrix element would have to be handled separately. We only consider the CM in what follows, although the presented ideas could be straightforwardly used for the SCEM as well.
	
	Our discretisation of $D\!\Lambda(A)$ is based on a partitioning of a subset $\widetilde{\Omega}\subseteq\Omega$ into measurable sets $\{\Omega_n\}_{n=1}^{\mathsf{N}}$ and the use of an orthonormal basis $\{f_j\}_{j\in\mathbb{N}}$ for $L^2_{\diamond}(\Gamma)$. One may,~e.g.,~have prior knowledge that the perturbation is supported at a distance from $\partial\Omega$, which can then be reflected in the choice of $\widetilde{\Omega}$. We make the assumption that $b$ is piecewise constant on the partition, meaning that $b|_{\Omega_n} \equiv b_n\in\mathbb{C}$ for $n=1,\dots,\mathsf{N}$. Thereby $\mathcal{W}$ is the vector space of piecewise constant functions adhering to the partition $\{\Omega_n\}_{n=1}^{\mathsf{N}}$ and extended by zero to the rest of $\Omega$. 
	
	Define $u_j = Nf_j$. In the considered case $A$ is obviously Hermitian, so we can use \eqref{eq:DLambdasimple} to obtain
	\begin{equation} \label{eq:Fdisk}
		\inner{D\!\Lambda(A;\chi_{\Omega_n}I)f_j,f_i}_{L^2(\Gamma)} = -\int_{\Omega_n} \nabla u_j\cdot \overline{\nabla u_i}\,\di x, \quad i,j \in \N.
	\end{equation}
	This gives a means for calculating an (infinite) matrix representation for the derivative with respect to $\chi_{\Omega_n}I$, i.e.\ with respect to the ``$n$'th pixel'', for $n=1,\dots,\mathsf{N}$. It is evident from \eqref{eq:Fdisk} that once all considered $u_j$ have been numerically approximated, e.g.,\ using a finite element method, the discretisation of $D\!\Lambda(A)$ is merely a question of numerical integration. In particular, as the approximations of $u_j$ are already needed for representing the ND map $\Lambda(A)$ itself, the extra computational cost in \eqref{eq:Fdisk} lies solely with the aforementioned numerical integration. See also \cite{Harrach2021} for additional details on such finite element implementations.

    In practice, one cannot consider an infinite number of boundary current densities, which we take into account by defining $\mathscr{P}$ to be the orthogonal projection onto the span of, say, the first $\mathsf{J}$ basis functions $\{f_j\}_{j=1}^\mathsf{J}$. In accordance with this idea, let $\mathsf{\Lambda}$ be the $\mathsf{J}\times\mathsf{J}$ matrix representation of the datum $\mathscr{P}\Lambda(A+B)\mathscr{P}$ with respect to $\{f_j\}_{j=1}^\mathsf{J}$, which in the following is also identified with a $\mathsf{J}^2$-column vector. Furthermore, let $\mathsf{D\Lambda}^{-1}$ be the Moore--Penrose pseudoinverse of the $\mathsf{J}^2\times\mathsf{N}$ matrix representation of $\mathscr{F} = \mathscr{P}D\!\Lambda(A;\,\cdot\,)\mathscr{P}$. If $\mathscr{F}$ is injective on $\mathcal{W}$, which may be achieved,~e.g.,~via choosing $\mathsf{J}$ to be large enough, then an application of $\mathsf{D\Lambda}^{-1}$ corresponds to first projecting onto $\mathcal{Y} = \mathscr{F}(\mathcal{W})$ and then applying $\mathscr{F}^{-1} : \mathcal{Y}\to\mathcal{W}$. The implementation of the projection $Q$ is therefore implicitly included in this construction.
		
    Once $F_1$ has been determined, the other $F_j$ can be found inductively. Indeed, by having a closer look at \eqref{eq:F2}--\eqref{eq:F4}, it becomes clear that once a (discretised) inverse of $\mathscr{F}$ is in hand, determining $F_2$ essentially only requires applications of $P(F_1)$, determining $F_3$ only requires applications of $P(F_1)$ and $P(F_2)$, and determining $F_4$ only requires applications of $P(F_1)$, $P(F_2)$, and $P(F_3)$; see Remark~\ref{rm:structure}. In particular, applying $P(F_1)$, $P(F_2)$, or $P(F_3)$ is cheaper than assembling and inverting the matrix representation of $\mathscr{F}$; as these operators share the sesquilinear form associated to the weak formulation of the CM (or SCEM), the matrix factorisation of the finite element system used in the computations of $\{u_j\}_{j=1}^\mathsf{J}$ may be reused. 
	
	To be more precise, $F_1$ can computed via
	\begin{equation*}
		F_1 = \mathsf{D\Lambda}^{-1}\Big( \mathsf{\Lambda} - \big[\inner{Tu_j,f_i}_{L^2(\Gamma)}\big]_{i,j=1}^\mathsf{J} \Big).
	\end{equation*}
	Subsequently, $F_2$ is given by
	\begin{align*}
		h_j &= -P(F_1) u_j, \quad j=1,\dots,\mathsf{J}, \\
		v_j &= P(F_1) h_j, \quad j=1,\dots,\mathsf{J}, \\
		F_2 &= \mathsf{D\Lambda}^{-1} \big[\inner{Tv_j,f_i}_{L^2(\Gamma)}\big]_{i,j=1}^\mathsf{J}.
	\end{align*}
	Next, it is the turn of $F_3$:
	\begin{align*}
		w_j &= -P(F_2) u_j, \quad j=1,\dots,\mathsf{J}, \\
		p_j &= P(F_2)h_j, \quad j=1,\dots,\mathsf{J}, \\
		q_j &= P(F_1)\left[v_j+w_j\right], \quad j=1,\dots,\mathsf{J}, \\
		F_3 &= \mathsf{D\Lambda}^{-1} \big[\inner{T(p_j+q_j),f_i}_{L^2(\Gamma)}\big]_{i,j=1}^\mathsf{J}.
	\end{align*}
	Lastly, we may compute $F_4$ as follows:
	\begin{align*}
		r_j &= -P(F_3)u_j, \quad j=1,\dots,\mathsf{J}, \\
		x_j &= P(F_3)h_j, \quad j=1,\dots,\mathsf{J}, \\
		y_j &= P(F_2)\left[ v_j + w_j \right], \quad j=1,\dots,\mathsf{J}, \\
		z_j &= P(F_1)\left[ p_j + q_j + r_j \right], \quad j=1,\dots,\mathsf{J}, \\
		F_4 &= \mathsf{D\Lambda}^{-1}\big[\inner{T(x_j+y_j+z_j),f_i}_{L^2(\Gamma)}\big]_{i,j=1}^\mathsf{J}.
	\end{align*}
	It is straightforward to verify that the above formulas are concordant with \eqref{eq:F2}--\eqref{eq:F4} in Theorem~\ref{thm:main}.
	
	To put the computational complexity of our construction into perspective, consider instead resorting to some higher order Newton-type numerical method for approximately solving Calder{\'o}n's problem. Then one would need, e.g.,\ a discretisation of the second derivative $D^2\!\Lambda(A)$, which in the context of isotropic coefficients corresponds to computing $\inner{D^2\!\Lambda(A;\chi_{\Omega_n}I,\chi_{\Omega_{\tilde{n}}}I)f_j,f_i}_{L^2(\Gamma)}$ for $n,\tilde{n}=1,\dots,\mathsf{N}$ and $i,j = 1,\dots,\mathsf{J}$.  The size of such a matrix representation grows by a factor $\mathsf{N}$ each time a higher order derivative is included in the analysis, quickly becoming infeasibly large for practical computations. On the other hand, the total computational complexity of our series reversion method only grows by a scalar multiple (independent of $\mathsf{N}$) when increasing its order.

	\begin{remark} \label{remark:regularisation}
          It should be emphasised that the {\em only} ill-conditioned steps in the proposed family of numerical methods are the applications of $\mathsf{D\Lambda}^{-1}$. For reconstruction from noisy measurements, or simply due to the inherent ill-conditioning of the involved computations, one arguably should usually resort to some regularised inverse when applying $\mathsf{D\Lambda}^{-1}$~\cite{Engl1996}. To this end, one can e.g.\ use a truncated singular value decomposition for computing a regularised version of $\mathsf{D\Lambda}^{-1}$, with all singular values below a given threshold $\alpha\geq 0$ set to zero before finding the Moore--Penrose inverse. In addition, one may set a lower limit  $\beta\geq 0$ for the pointwise contrast required from the reconstructed perturbation via replacing $F_j$ (before continuing to compute $F_{j+1}$) by the expression
		\begin{equation*}
			\tau_\beta \biggl(\sum_{k=1}^j F_k\biggr) - \sum_{k=1}^{j-1}F_k.
		\end{equation*}  
		Here, the cut-off function $\tau_\beta$ is applied pointwise, with $\tau_\beta(C(x)) = C(x)$ if $\norm{C(x)}_2 \geq \beta$ and $\tau_\beta(C(x)) = 0$ otherwise.
	\end{remark}

    \subsection{First numerical test}
        
    In the first numerical test, we present a few simple numerical reconstructions of a certain isotropic perturbation $B$ in a two-dimensional unit disk $\Omega$; see Figure~\ref{fig:fig1}. The perturbation takes the value $0.3$ in a square-shaped inclusion and the value $0.8$ in a pentagon-shaped inclusion, and it vanishes in the rest of $\Omega$. In all examples, $A$ is the identity matrix, corresponding to an isotropic unit background conductivity in $\Omega$.
        
    \begin{figure}[!htb]
    	\includegraphics[width=.28\linewidth]{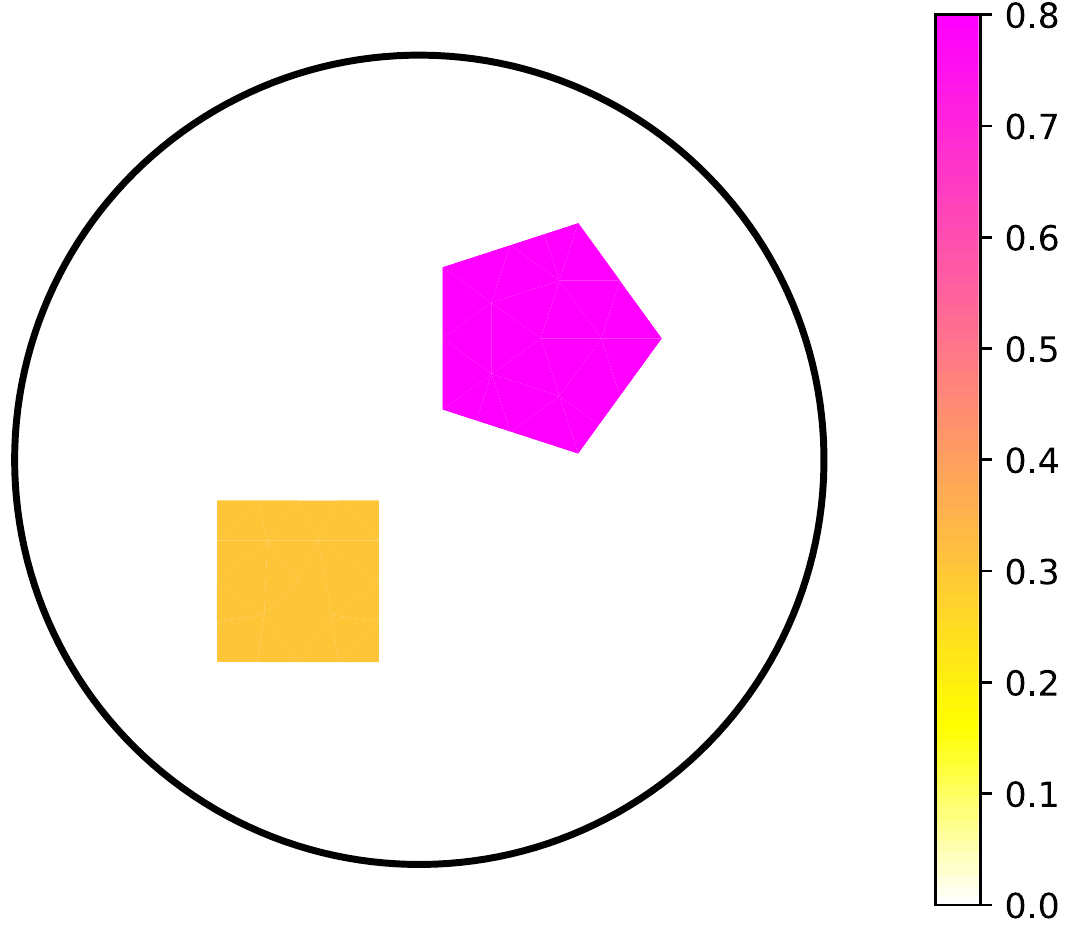}
    	\caption{Isotropic perturbation $B$.} \label{fig:fig1}
    \end{figure}

    We employ the first $\mathsf{J} = 20$ orthonormal Fourier basis functions of $L^2_\diamond(\partial\Omega)$ as the Neumann boundary values, i.e.~as the input boundary current densities. A finite element method with $\mathbb{P}_3$-elements is used for numerically solving all involved variational problems. A fine triangular mesh, for which the inclusion boundaries are aligned with edges of the elements, is used for simulating the datum~$\mathsf{\Lambda}$. Another fine mesh is employed for solving the variational problems involved in the series reversion. However, a much coarser mesh and the associated triangle-wise piecewise constant basis functions define the discretisation for $\mathcal{W}$, which contains $F_j$, $j=1,\dots,4$, and serves as the codomain for $\mathsf{D\Lambda^{-1}}$. In terms of regularisation, we follow Remark~\ref{remark:regularisation} with $\alpha = 3\cdot 10^{-5}$ and $\beta = 10^{-1}$. Moreover, it is assumed to be {\em a priori} known that the distance from the support of $B$ to $\partial \Omega$ is at least $0.15$, that is, $\widetilde{\Omega}$ is an origin-centered disk of radius $0.85$. All presented reconstructions correspond to noiseless data (not accounting for numerical inaccuracies).

	\begin{figure}[!htb]
		\includegraphics[width=\linewidth]{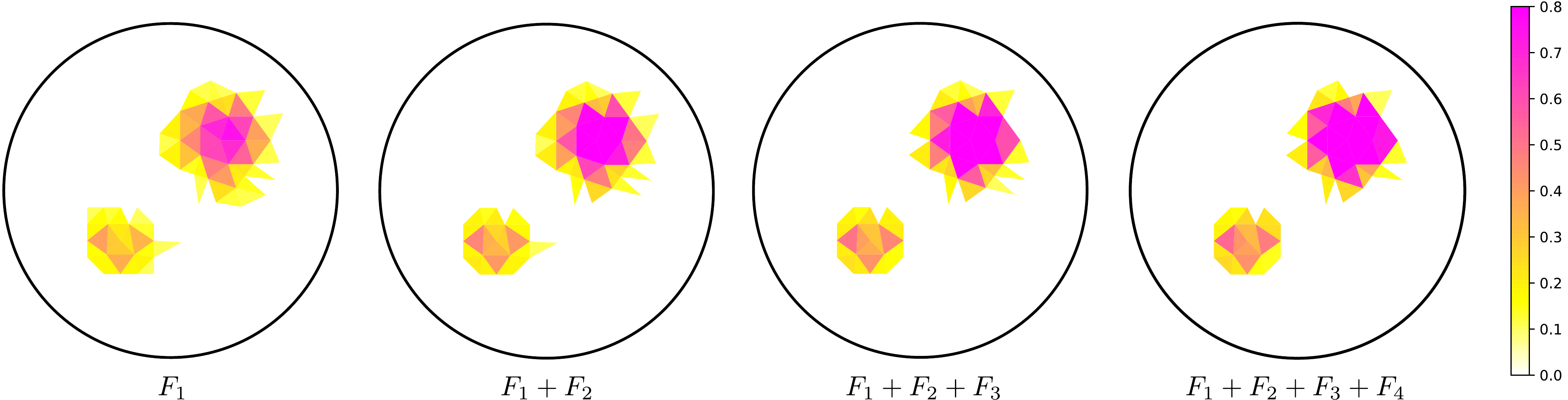}
		\caption{Approximations of $B$ by series reversions of different orders with $B\in\mathcal{W}$.} \label{fig:fig2}
	\end{figure}
	
	\begin{figure}[!htb]
		\includegraphics[width=\linewidth]{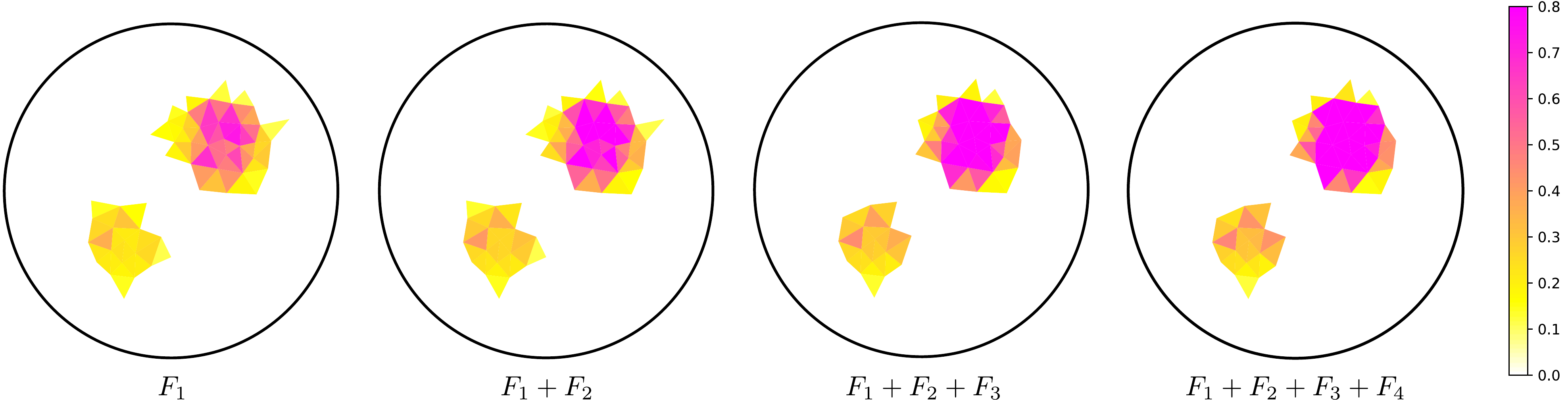}
		\caption{Approximations of $B$  by series reversions of different orders with $B\not\in\mathcal{W}$.} \label{fig:fig3}
	\end{figure}

    Figure~\ref{fig:fig2} shows the four reconstructions defined by $F_1, \dots, F_4$ on a reconstruction mesh aligned with the unknown inclusion boundaries, meaning that $B\in\mathcal{W}$. The reconstructions in Figure~\ref{fig:fig3} correspond to a more realistic and practically relevant setting, where the mesh is not aligned with the inclusion boundaries, i.e.\ $B\not\in\mathcal{W}$. In both figures, the quality of the presented reconstructions increases as more terms are included in \eqref{eq:Bseries}. Not surprisingly, the reconstructions in Figure~\ref{fig:fig2} more accurately capture the shapes of the inclusions defining $B$.

    On a standard laptop and the denser mesh used for solving the variational problems, $F_1$ is formed on average in 14.17 seconds, with the majority of the time spent on assembling the Fr\'echet derivative and factorising the finite element system matrix. The computation times of the subsequent terms $F_2$, $F_3$, and $F_4$ are 2.40, 3.51, and 4.72 seconds, respectively. The total time for acquiring $F_4$ is thus 24.80~seconds. Using the coarse reconstruction mesh for also solving the variational problems during the reconstruction process, resulting only in slight changes to the reconstructions, the respective computation times for forming $F_1$, $F_2$, $F_3$, and $F_4$ are 11.30, 0.39, 0.50, and 0.61~seconds.

    \subsection{Second numerical test}

    The main objective of the second numerical test is to verify the convergence rates predicted by Theorem~\ref{thm:main} in a simple radially symmetric geometry. In particular, we avoid considering the sources of numerical errors, stemming from,~e.g.,~the choice of a finite element method, by using a more direct implementation in the chosen simple setting. To this end, let $\Omega\subset \R^2$ again be the open unit disk, let the constant background conductivity $A$ still equal the identity matrix $I$, and suppose the perturbation $B$ is of the form
    \begin{equation} \label{eq:conB}
    	B = \begin{dcases}
      		\kappa_1 I &  \text{in} \ \Omega \setminus \overline{D_{\rho}}, \\ 
      		\kappa_2 I &  \text{in} \ D_{\rho},
    	\end{dcases}
    \end{equation}
    for scalars $\kappa_1,\kappa_2\in (-1,\infty)$ and where $D_\rho$ is a smaller open origin-centered disk of radius $\rho \in(0,1)$. We assume to know the radius $\rho$ and thus choose the subspace $\mathcal{W}$, from which the reconstruction is sought, to be
    \begin{equation*}
    	\mathcal{W} = \big\{ \eta_1 \chi_1 I +  \eta_2 \chi_2 I \mid \eta_1, \eta_2 \in \C \big\},
    \end{equation*}
    where $\chi_1$ and $\chi_2$ are the characteristic functions of $\Omega \setminus \overline{D_{\rho}}$ and $D_{\rho}$, respectively. The elements of $\mathcal{W}$ can obviously be identified with vectors of $\C^2$.

    Let $f_j= \frac{1}{\sqrt{2 \pi}} {\rm e}^{{\rm i}j \, \cdot}$, $j \in \Z' = \Z \setminus \{ 0 \}$, denote the orthonormal Fourier basis functions for $L^2_\diamond(\partial \Omega)$ parametrised with respect to the polar angle. Straightforward calculations reveal that the ND maps $\Lambda(A)$ and $\Lambda(A + B)$ are characterised by the spectral decompositions
    $$
    \Lambda(A): f_j \mapsto \frac{1}{|j|} f_j
    $$
    and
    \begin{equation}
      \label{eq:FApBND}
    \Lambda(A + B): f_j \mapsto \frac{1}{(\kappa_1 + 1)|j|}  \, \frac{(\kappa_1 + \kappa_2 + 2) - (\kappa_2-\kappa_1)\rho^{2|j|}}{(\kappa_1 + \kappa_2 + 2) + (\kappa_2-\kappa_1)\rho^{2|j|}} \, f_j
    \end{equation}
    for $j \in \Z'$. It is also well known that
    \begin{equation}
      \label{eq:FNeumann}
    N(A)f_j(r, \theta) = \frac{1}{\sqrt{2 \pi}|j|} \,  r^{|j|} \, {\rm e}^{{\rm i}j \theta}, \qquad j \in \Z',
    \end{equation}
    where $(r, \theta)$ are the polar coordinates in $\Omega$.

    Let us next consider $P(\eta)$, where $\eta \in \C^2$ is identified with the corresponding element of $\mathcal{W}$. As the conductivity perturbations in $\mathcal{W}$ equal scalar multiples of the identity matrix in both $\Omega \setminus \overline{D_{\rho}}$ and $D_{\rho}$, we only need to consider how $P(\eta)$ operates on elements of
    \begin{equation*}
    	\mathscr{H}(\Omega) = \big \{ v \in H^1_\diamond(\Omega) \ \big| \ \Delta v = 0 \ \text{in} \ \Omega \setminus \partial D_{\rho} \}  
    \end{equation*}
    since the range of $N(A)$ is obviously contained in $\mathscr{H}(\Omega)$ and $P(\eta)(\mathscr{H}(\Omega)) \subset \mathscr{H}(\Omega)$ for any $\eta \in \C^2$, as easily follows from the definition of $P$ in~\eqref{eq:defP}. The space $\mathscr{H}(\Omega)$ is composed of the functions
    \begin{equation} \label{eq:scrH}
    	v(r, \theta) = \begin{dcases}
      {\displaystyle \sum_{j \in \Z'}} \big( \alpha_j  r^{|j|} + \beta_j r^{-|j|} \big) {\rm e}^{{\rm i}j \theta} &  \text{in} \ \Omega \setminus \overline{D_{\rho}}, \\
     {\displaystyle \sum_{j \in \Z'}} \gamma_j  r^{|j|} {\rm e}^{{\rm i}j \theta}  & \text{in} \ D_{\rho},
    \end{dcases}
    \end{equation}
    with coefficients $\alpha_j, \beta_j, \gamma_j \in \C$ that satisfy the conditions
    \begin{equation} \label{eq:other_two_conditions}
    	\sum_{j \in \Z'} |j| \big( |\alpha_j|^2 + |\beta_j|^2) < \infty \qquad \text{and} \qquad
    	\alpha_j + \rho^{-2|j|} \beta_j = \gamma_j, \quad j \in \Z'.
    \end{equation}
    %
    The first condition in \eqref{eq:other_two_conditions} is equivalent to requiring that $v|_{\Omega \setminus \overline{D_{\rho}}} \in H^1(\Omega \setminus \overline{D_{\rho}})$, and the second condition
    is equivalent to the Dirichlet trace of $v$ being continuous over $\partial D_{\rho}$. The two conditions in \eqref{eq:other_two_conditions} also guarantee that $v|_{D_\rho} \in H^1(D_{\rho})$.

    In consequence, it is sufficient to understand how $P(\eta)$ operates on functions satisfying \eqref{eq:scrH} and \eqref{eq:other_two_conditions}. Via a tedious but straightforward calculation based on the definition of $P$, it follows that an arbitrary element of $\mathscr{H}(\Omega)$, defined through its coefficients $\{(\alpha_j, \beta_j, \gamma_j)\}_{j\in \Z'}$, is mapped by $P(\eta)$ to another element of $\mathscr{H}(\Omega)$ whose coefficients $\{ (\widetilde{\alpha}_j, \widetilde{\beta}_j, \widetilde{\gamma}_j)\}_{j \in \Z'}$ are given by
    \begin{equation} \label{eq:Drecursion}
    \begin{bmatrix}
    \widetilde{\alpha}_j\\
    \widetilde{\beta}_j \\
    \widetilde{\gamma}_j
    \end{bmatrix}
    =
    \frac{1}{2}\left( \eta_1 \begin{bmatrix}
    \rho^{2|j|} - 2 & 1 & 0 \\  \rho^{2|j|} & -1 & 0 \\ \rho^{2|j|} - 1 & 1-\rho^{-2|j|} & 0
    \end{bmatrix}
    - \eta_2 \, \rho^{2|j|}
    \begin{bmatrix}
    0 & 0 & 1 \\
    0 & 0 & 1 \\
    0 & 0 & 1 + \rho^{-2|j|}
    \end{bmatrix} \right)
    \begin{bmatrix}
    \alpha_j\\
    \beta_j  \\
    \gamma_j
    \end{bmatrix}
    \end{equation}
    for all $j \in \Z'$. Note that the vectors in the range of the matrix on the right-hand side of \eqref{eq:Drecursion} automatically satisfy the second condition of \eqref{eq:other_two_conditions}.

    Finally, as $T$ maps a function of the form \eqref{eq:scrH} to
    \begin{equation} \label{eq:Ftrace}
    	\sum_{j \in \Z'} (\alpha_j + \beta_j){\rm e}^{{\rm i}j \, \cdot} \in L^2_\diamond(\partial \Omega),
    \end{equation}
    we have introduced all necessary tools for implementing the approximation formulas of Theorem~\ref{thm:main} in our simple concentric geometry. In particular, all terms appearing on the right-hand side of \eqref{eq:F2}--\eqref{eq:F4} are obtained by applying the operator $M$, defined by the inverse of the projected derivative $\mathscr{F}$, to operators that admit spectral decompositions with respect to the Fourier basis $\{f_j\}_{j \in \Z'}$. The derivative $D\!\Lambda(A; \eta)$ itself also obeys such a spectral decomposition for any $\eta \in \C^2$: based on \eqref{eq:Nderiv}, \eqref{eq:FNeumann}, \eqref{eq:Drecursion}, and \eqref{eq:Ftrace},
    \begin{equation*}
    	D\!\Lambda(A; \eta): f_j \mapsto \frac{1}{|j|}\left( \eta_1 \big(\rho^{2 |j|} - 1 \big) - \eta_2 \, \rho^{2 |j|} \right) \! f_j,
    	\qquad j \in \Z',
    \end{equation*}
    which can also be verified using \eqref{eq:FApBND}. Hence, if $\mathscr{P}$ is chosen to be a projection onto the span of some Fourier basis functions --- as it will be in what follows --- one can omit its left-hand occurrences in the definitions of both $\mathscr{F}$ and $M$.
	
	\begin{example}
	
    Let us then present the actual numerical examples. Consider first $\rho = 0.3$ and the cases when either $\kappa_1$ or $\kappa_2$ is {\em a priori} known to be zero. Because there is only one parameter to be reconstructed, we also choose to employ only one current pattern and select $\mathscr{P}$ to be the orthogonal projection onto $\mspan\{f_1\}$. The left and right images in Figure~\ref{fig:conv1} show the signed errors 
    \begin{equation*}
    	\kappa_1- \sum_{k=1}^K F_k \quad \textup{and} \quad \kappa_2 - \sum_{k=1}^K F_k
    \end{equation*}
    as functions of $\kappa_1$ and $\kappa_2$, respectively, over the interval $[-0.5, 1]$. The approximation becomes uniformly better as $K$ increases. An interesting observation is that for negative perturbations the approximations resemble Taylor polynomials in the sense that the sign of the error alternates as a function of $K$, but for positive perturbations the signed error remains positive.

    \begin{figure}[htb]
		\includegraphics[width=0.485\linewidth]{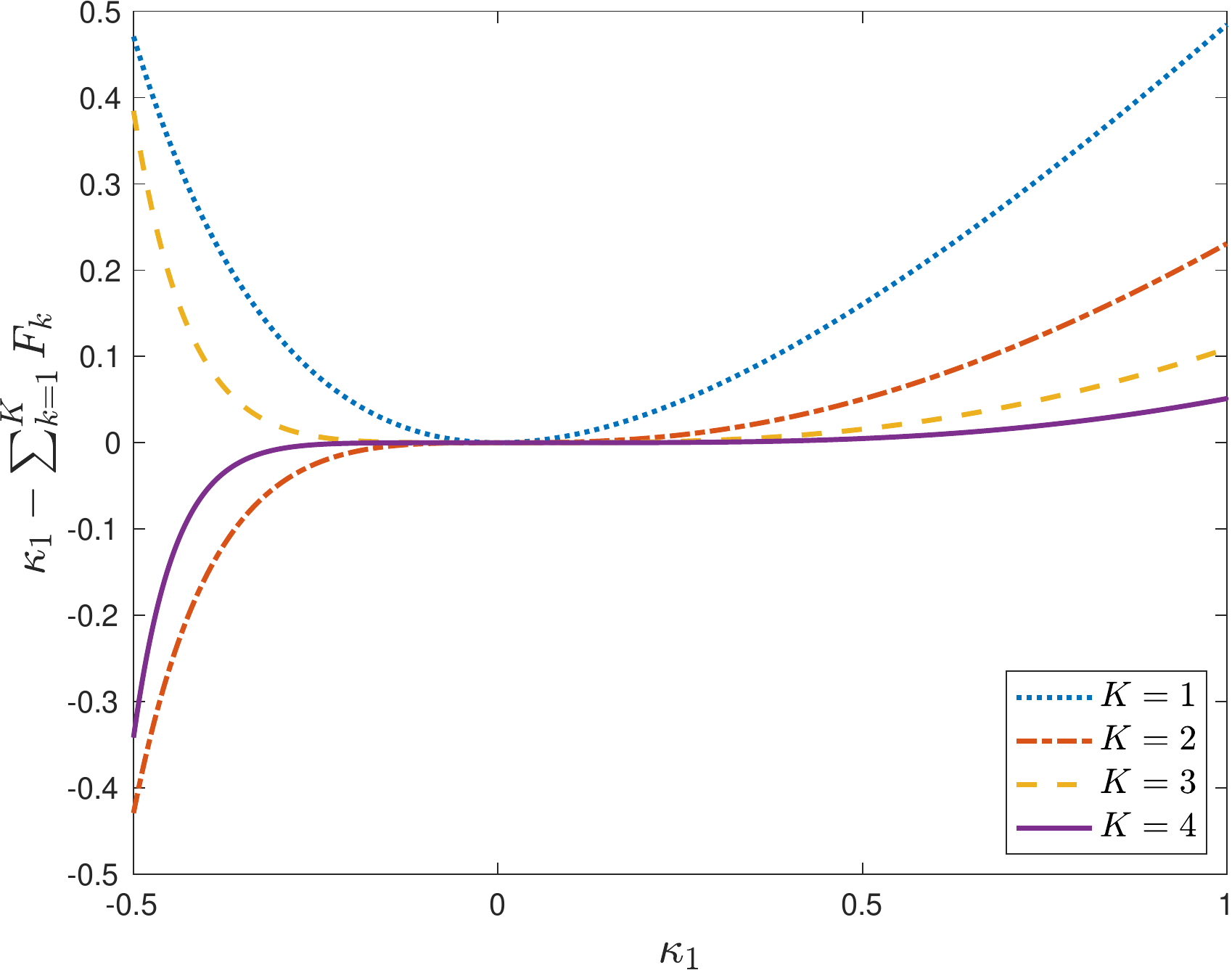} \hfill
        \includegraphics[width=0.485\linewidth]{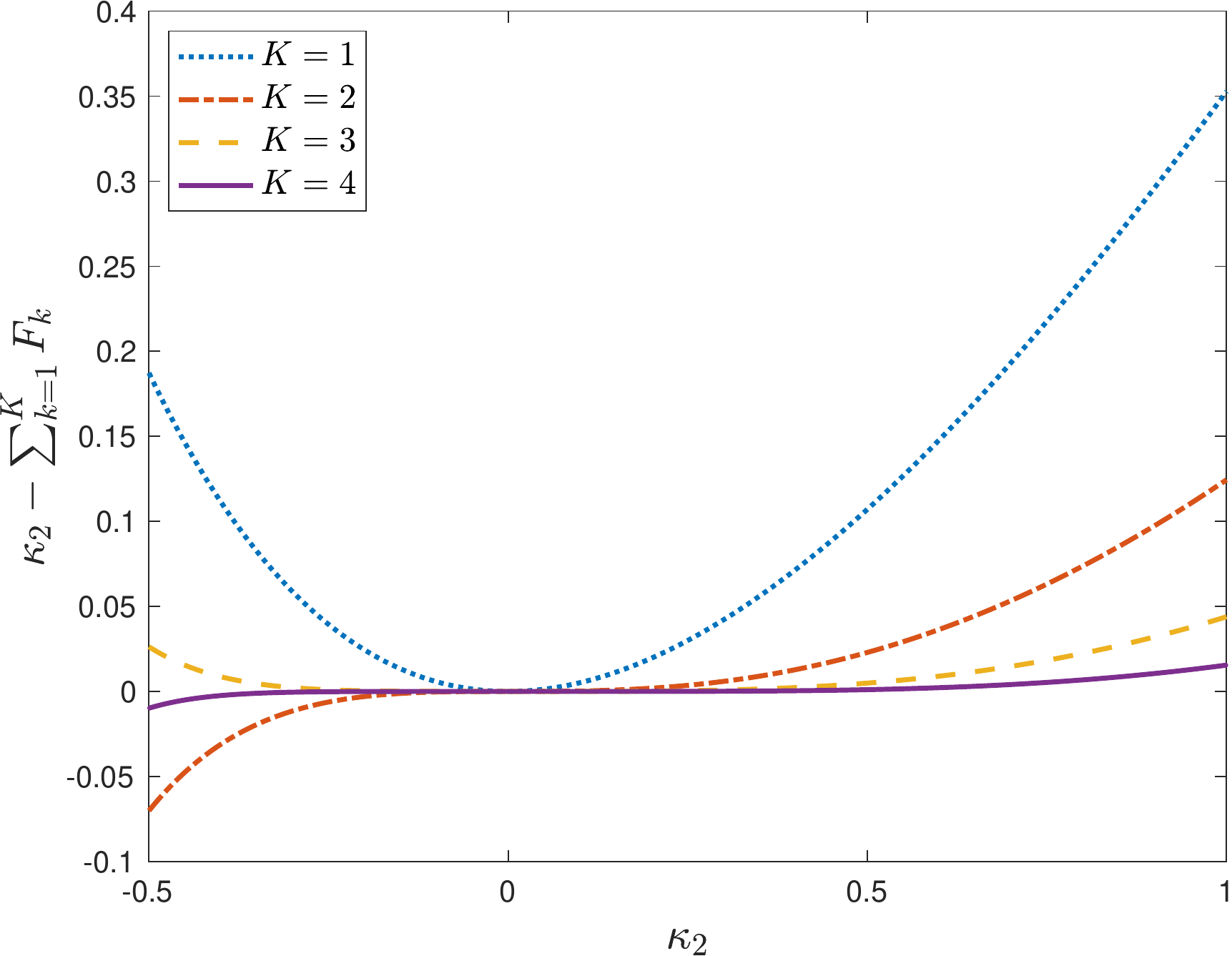}
		\caption{The employed current density is $f_1$ and  the radius of the inner disk is $\rho = 0.3$. Left: The signed errors $\kappa_1 - \sum_{k=1}^K F_k$, $K=1, \dots, 4$, as functions of $\kappa_1$ when it is {\em a priori} known that $\kappa_2=0$. Right: The signed errors $\kappa_2 - \sum_{k=1}^K F_k$, $K=1, \dots, 4$, as functions of $\kappa_2$ when it is {\em a priori} known that $\kappa_1=0$.} \label{fig:conv1}
	\end{figure}

	\end{example}
	
	\begin{example}
	
    Next we let $\rho = 1/\sqrt{2}$, so that the areas of $\Omega \setminus \overline{D_{\rho}}$ and $D_{\rho}$ are the same,~i.e.~$\pi/2$. We now consider reconstructing both parameters in $\kappa = (\kappa_1,\kappa_2)$. As there are two free parameters, it is natural to also employ two current patterns, and thus we choose $\mathscr{P}$ to be the orthogonal projection onto $\mspan\{f_1,f_2\}$. 
    
    \begin{figure}[htb]
    	\includegraphics[width=0.485\linewidth]{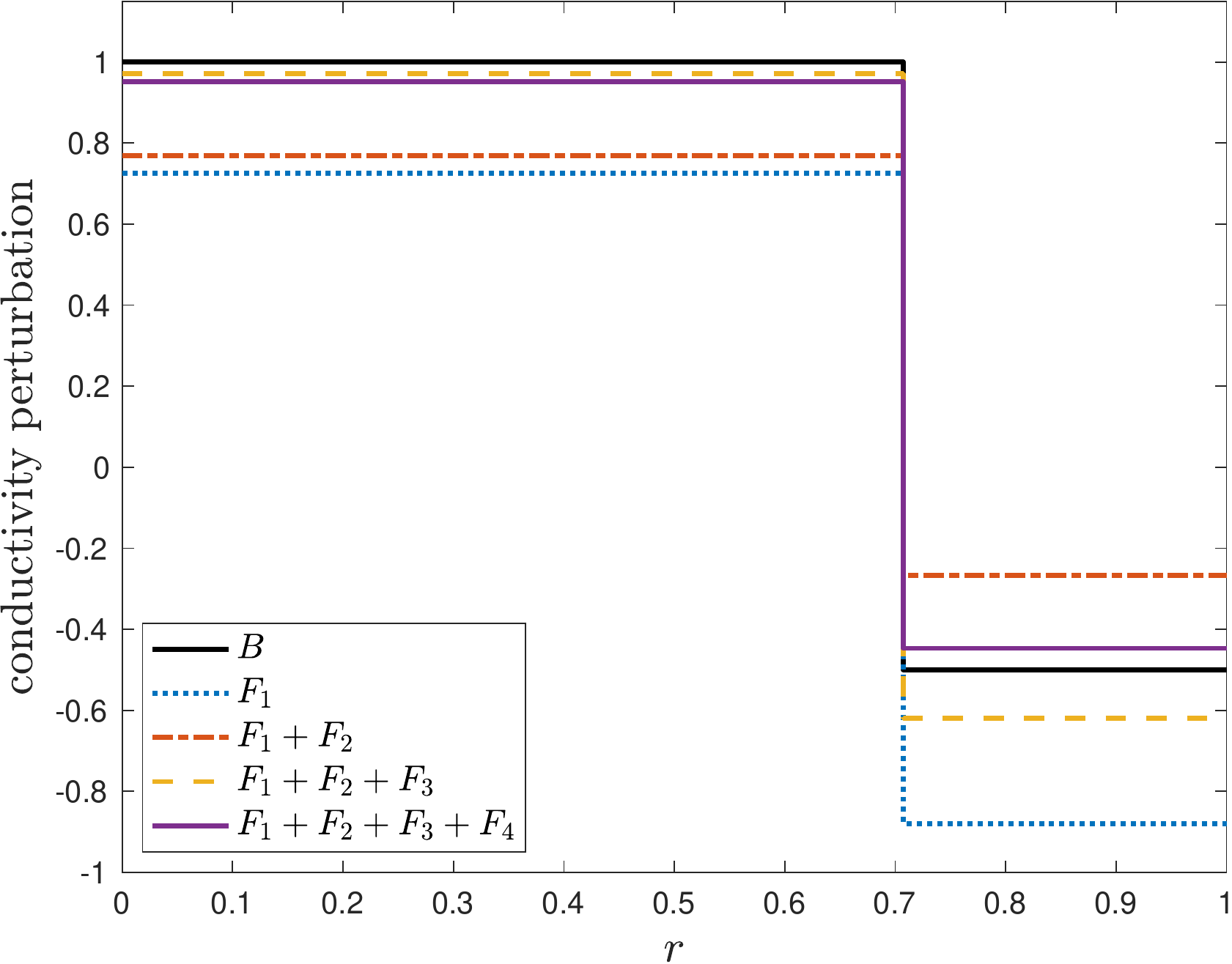} \hfill
    	\includegraphics[width=0.485\linewidth]{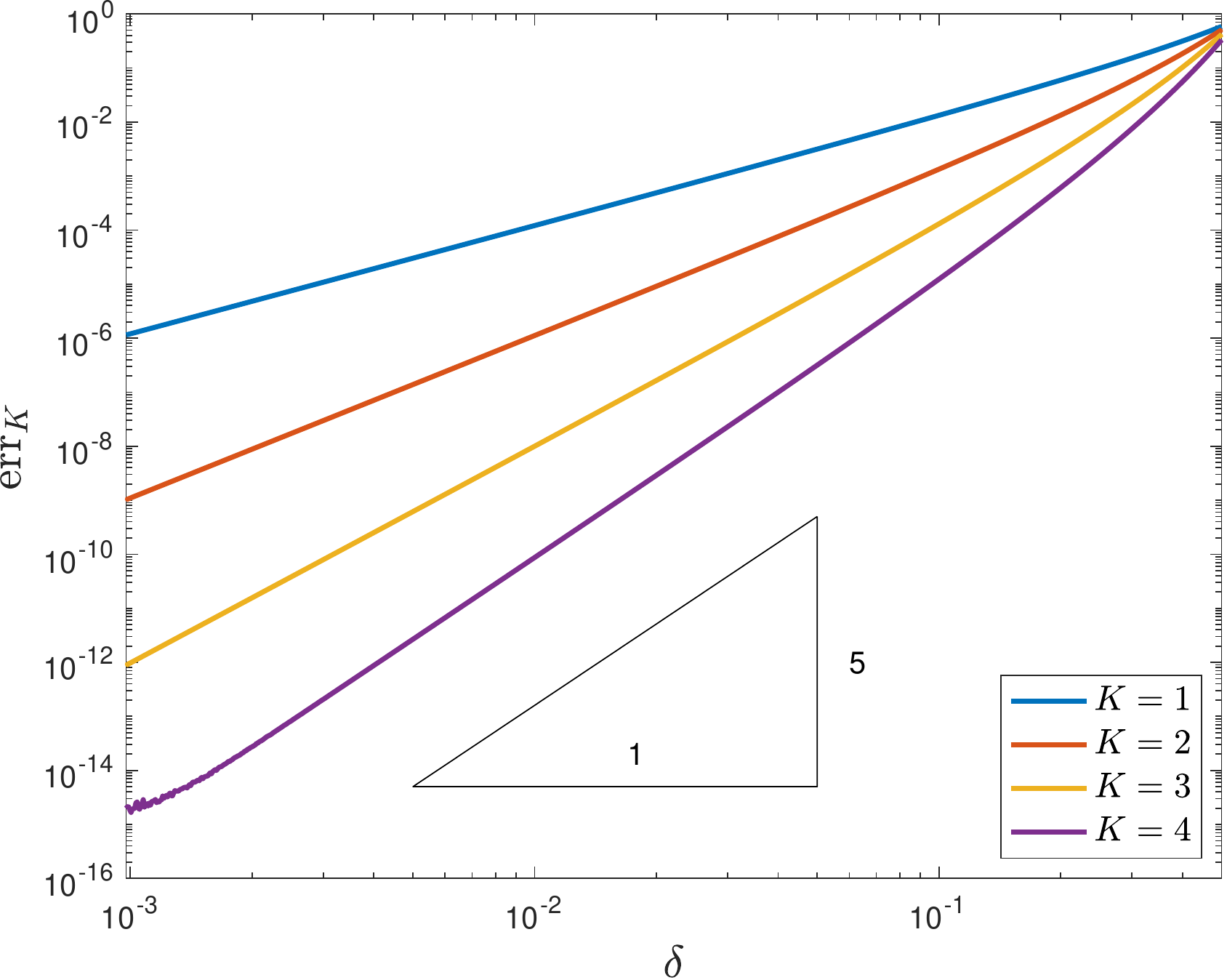}
    	\caption{The employed current densities are $f_1$ and $f_2$, and  the radius of the inner disk is $\rho = 1/\sqrt{2}$. Left: The isotropic perturbation $B$, with $\kappa = (-0.5,1)$, and the corresponding reconstructions $\sum_{k=1}^K F_k$, $K=1,\dots,4$, as functions of radial variable $r$. Right: The maximal $L^2(\Omega)$ reconstruction error ${\rm err}_K$, defined by \eqref{eq:max_error}, as a function of the Euclidean norm of the perturbations $\delta = |\kappa|$.} \label{fig:conv2}
    \end{figure}
    
    The left-hand image of Figure~\ref{fig:conv2} compares the reconstructions $\sum_{k=1}^K F_k$, as functions of the radial variable $r$, with the corresponding exact conductivity perturbation $B$ defined by $\kappa = (-0.5,1)$ in \eqref{eq:conB}. The reconstruction error over the whole of $\Omega$ decays as $K$ increases, but not monotonically in the interior disk $D_\rho$ when moving from $K=3$ to $K=4$. For completeness it should also be mentioned that the approximations provided by Theorem~\ref{thm:main} seem to diverge as $K$ increases if the target perturbation $B$ is significantly larger,~e.g.,~if it is defined by $\kappa = (-0.75,2)$. 

    The right-hand image of Figure~\ref{fig:conv2} depicts the maximal $L^2(\Omega)$ reconstruction errors
    \begin{equation} \label{eq:max_error}
    	{\rm err}_K(\delta) = \sqrt{\frac{\pi}{2}} \ \max_{ |\kappa| = \delta } \Big| \kappa - \sum_{k=1}^K F_k \Big|, \qquad K=1, \dots, 4,
   	\end{equation}
    as functions of the Euclidean norm of the perturbations $\delta$, where $\kappa$ and $F_k$ are interpreted as elements of $\C^2$ on the right-hand side. The convergence rates predicted by \eqref{eq:Bseries} are clearly visible in the right-hand image of Figure~\ref{fig:conv2}. This numerically verifies Theorem~\ref{thm:main} in the considered simple setting, as the $L^2$-norm is equivalent to the $L^\infty$-norm for piecewise constant functions on a fixed partition of a bounded domain. The source of the wriggles at the very left end of the curve corresponding to $K=4$ is presumably floating point accuracy.
	
	\end{example}
         
    \subsection*{Acknowledgments}

    This work is supported by the Academy of Finland (decision 336789) and the Aalto Science Institute (AScI). In addition, HG is supported by The Research Foundation of DPhil Ragna Rask-Nielsen and is associated with the Aarhus University DIGIT Centre, and NH is supported by Jane and Aatos Erkko Foundation via the project Electrical impedance tomography --- a novel method for improved diagnostics of stroke.
        
    \appendix

    \section{$Q$ as an orthogonal projection} \label{sec:HS}
        
    If $\mathcal{Y} = \mathscr{F}(\mathcal{W})$ is finite-dimensional, as is always the case in numerical considerations as well as for the SCEM, one can avoid explicitly applying a projection $Q$ onto $\mathcal{Y}$ in the definition of $F_1$ via replacing $\mathscr{F}^{-1}: \mathcal{Y} \to \mathcal{W}$ by the corresponding Moore--Penrose pseudoinverse. Be that as it may, it is also interesting to consider equipping a suitable space of linear operators on $L^2_\diamond(\Gamma)$ with an inner product, which immediately leads to a systematic way of projecting onto any closed subspace. To this end, we restrict our attention to the CM in a two-dimensional bounded simply connected domain $\Omega \subset \R^2$ with a $C^{1, \alpha}$ boundary for some $\alpha > 0$. It turns out that under these assumptions $Q$ can be chosen to be the orthogonal projection onto $\mathcal{Y}$ in the Hilbert space of Hilbert--Schmidt operators on $L^2_\diamond(\Gamma)$, assuming $\Gamma = \partial \Omega$. Most of the results presented in this appendix do not hold in higher spatial dimensions, but the generalisation to the case of partial data, i.e.~$\Gamma\not = \partial \Omega$, is not explicitly excluded.
        
    Let $H_1$ and $H_2$ be separable Hilbert spaces and recall that a compact linear operator $L : H_1 \to H_2$ belongs to the Schatten class $\mathscr{L}_p(H_1,H_2)$, $1 \leq p < \infty$, if its nonincreasing sequence of singular values $\{\sigma_{L,k}\}_{k \in \N}\subset\mathbb{R}_+$ belongs to $\ell^p(\mathbb{N})$. The upper limit $\mathscr{L}_\infty(H_1,H_2)$ is identified with the space of compact operators $\mathscr{L}_{\rm C}(H_1,H_2)$. The class $\mathscr{L}_p(H_1,H_2)$, $1 \leq p < \infty$, becomes a Banach space when equipped with the norm
    \begin{equation} \label{eq:Schatten_norm}
	    \norm{L}_{\mathscr{L}_p(H_1,H_2)}^p = \sum_{k=1}^\infty {\sigma_{L,k}}^p, \qquad L \in \mathscr{L}_p(H_1,H_2).
    \end{equation}
    The spaces $\mathscr{L}_1(H_1,H_2)$ and $\mathscr{L}_2(H_1,H_2)$ are called the trace-class operators and the Hilbert--Schmidt operators, respectively. The class of Hilbert--Schmidt operators $\mathscr{L}_2(H_1,H_2)$ is a Hilbert space when equipped with the inner product
    \begin{equation} \label{eq:HS_inner}
	    \langle L_1, L_2 \rangle_{\mathscr{L}_2(H_1,H_2)} = \sum_{k=1}^\infty \langle L_1 \phi_k, L_2 \phi_k \rangle_{H_2},
    \end{equation}
    where $\{\phi_j\}_{j \in\N}$ is an orthonormal basis for $H_1$. The definition \eqref{eq:HS_inner} can be shown to be independent of the choice of the orthonormal basis; choosing $\{\phi_j\}_{j \in\N}$ to be an orthonormal eigenbasis for $L^* L$, with $L_1=L$ and $L_2=L$, demonstrates that the inner product \eqref{eq:HS_inner} is concordant with the definition of the Schatten norm \eqref{eq:Schatten_norm} when $p=2$.~\cite[Chapter~7.1]{Weidmann1980}
         
    In the following $I \in \R^{2 \times 2}$ denotes the identity matrix, and $\Lambda(I)$ is thus the ND map for the unit conductivity.
    \begin{lemma}  \label{lemma:LHS}
 		Assume $\Gamma = \partial \Omega$ and that $\Omega \subset \R^2$ is bounded, simply connected, and has a $C^{1, \alpha}$ boundary for $\alpha > 0$. Then  $\Lambda(I): L^2_\diamond(\partial\Omega) \to L^2_\diamond(\partial\Omega)$ belongs to the Schatten class $\mathscr{L}_p(L^2_\diamond(\partial\Omega))$ for any $p>1$. In particular, $\Lambda(I)$ is a Hilbert--Schmidt operator, but it is not trace-class.
    \end{lemma}
    \begin{proof}
    Let $D \subset \R^2$ be the open unit disk and $\Phi: \Omega \to D$ a bijective conformal mapping. Since $\partial \Omega$ is of class $C^{1, \alpha}$, the map $\Phi$ extends to a bijective map of $\overline{\Omega}$ onto $\overline{D}$. Moreover, the extension of $\Phi$ onto $\partial \Omega$ and that of its inverse $\Psi$ onto $\partial D$ are continuously differentiable \cite[Theorem~3.6 \& Exercise~3.3.5]{Pommerenke1992}.

    It follows from,~e.g.,~\cite[Lemma~4.1 \& Remark~4.1]{Hyvonen_2018} that $\Lambda(I)$ can be factored as
    \begin{equation*}
        \Lambda(I) = \mathcal{P} \mathcal{C}_\Phi \widetilde{\Lambda}(I) \mathcal{M}_{\Psi},
    \end{equation*}
    where $\mathcal{P}$ is the orthogonal projection of $L^2(\partial \Omega)$ onto $L^2_\diamond(\partial \Omega)$,
    \begin{equation*}
        \mathcal{C}_{\Phi}:
        	\left\{
        	\begin{array}{l}
        		g \mapsto g \circ \Phi, \\[1mm]
        		L^2(\partial D) \to L^2(\partial \Omega),
        	\end{array}
        	\right.
    \end{equation*}
    and
    \begin{equation*}
        \mathcal{M}_{\Psi}:
        	\left\{
        	\begin{array}{l}
        		g \mapsto |\Psi'| (g\circ \Psi), \\[1mm]
        		L^2_\diamond(\partial \Omega) \to L^2_\diamond(\partial D).
        	\end{array}
        	\right.
    \end{equation*}
	Furthermore, $\widetilde{\Lambda}(I): L^2_\diamond(\partial D) \to L^2_\diamond(\partial D)$ is the ND map for the unit disk with unit conductivity, known to admit the spectral representation
    \begin{equation} \label{eq:D_spectral}
	    \widetilde{\Lambda}(I): \varphi_j \mapsto \frac{1}{|j|}\varphi_j, \quad j \in \Z \setminus \{ 0 \},
    \end{equation}
    where $\{ \varphi_j \}_{j \in \Z \setminus \{ 0 \}}$ is the standard complex orthonormal Fourier basis for $L^2_\diamond(\partial D)$.

    Since $\widetilde{\Lambda}(I): L^2_\diamond(\partial D) \to L^2_\diamond(\partial D)$ is self-adjoint, \eqref{eq:D_spectral} yields
    \begin{equation*}
        \norm{\widetilde{\Lambda}(I)}_{\mathscr{L}_p(L^2_\diamond(\partial \Omega))}^p =  2 \sum_{j=1}^\infty \frac{1}{j^p} < \infty, \quad p > 1,
    \end{equation*}
    and thus $\widetilde{\Lambda}(I) \in \mathscr{L}_p(L^2_\diamond(\partial D))$ for any $p > 1$. Since $\mathcal{M}_\Psi$, $\mathcal{C}_\Phi$, and $\mathcal{P}$ are bounded due to our assumptions on $\Omega$, the assertion follows from basic properties of Schatten operators~\cite[Theorem~7.8(c)]{Weidmann1980}.
	\end{proof}

    Take note that the above lemma fails in higher spatial dimensions as the eigenvalues of the ND map for the unit ball with unit conductivity are not even square-summable if $d>2$. Indeed, although the eigenvalues of such an ND map are still $1/j$ for $d > 2$, their multiplicity grows at least linearly in $j$ with the exact rate depending on $d$; cf.,~e.g.,~\cite[Proposition~3.3]{Garde2020}.
        
    The following theorem employs the mean free Sobolev spaces $H^s_\diamond(\partial\Omega)$ defined in \eqref{eq:Hsdiamond}.
    \begin{theorem} \label{thm:appendix}
	    Assume that $\Omega\subset\mathbb{R}^2$ is bounded, simply connected, and has a $C^{1, \alpha}$ boundary for $\alpha > 0$. If $L\in\mathscr{L}(H^{-s}_\diamond(\partial \Omega),H^{s}_\diamond(\partial \Omega))$ for some $s\in (0,\tfrac{1}{2}]$, then
        \begin{equation*}
	    	\norm{L}_{\mathscr{L}_p(L^2_\diamond(\partial \Omega))} \leq C \norm{L}_{\mathscr{L}(H^{-s}_\diamond(\partial \Omega),H^{s}_\diamond(\partial \Omega))},  \qquad p > \frac{1}{2s},
        \end{equation*}
        where $C = C(p, s, \Omega) > 0$ is independent of $L$. In particular, $L: L^2_\diamond(\partial \Omega) \to L^2_\diamond(\partial \Omega)$ is a Hilbert--Schmidt operator if $s>\tfrac{1}{4}$.
     \end{theorem}

     \begin{proof}
     	As in Lemma~\ref{lemma:LHS}, let $\Lambda(I): L^2_\diamond(\partial \Omega) \to L^2_\diamond(\partial \Omega)$ be the ND map corresponding to unit conductivity in $\Omega$ and $\Gamma = \partial\Omega$.  Denote by $\{ \phi_j \}_{j\in \N} \subset L^2_{\diamond}(\partial \Omega)$ eigenfunctions for $\Lambda(I)$, forming an orthonormal basis for $L^2_{\diamond}(\partial \Omega)$, and let $\{ \lambda_j\}_{j\in \N} \subset \R_+$ be the corresponding eigenvalues. By virtue of Lemma~\ref{lemma:LHS}, \eqref{eq:Schatten_norm}, and the self-adjointness of $\Lambda(I)$,
        \begin{equation} \label{eq:LSchatten}
		    \norm{\Lambda(I)}_{\mathscr{L}_t(L^2_\diamond(\partial \Omega))}^t = \sum_{j=1}^\infty \lambda_j^t < \infty
        \end{equation}
        for $t>1$. Moreover, notice that
        \begin{equation*}
		    \norm{f}_{r}^2 = \sum_{k=1}^\infty \lambda_k^{-2r}\abs{\inner{f, \phi_k}_{\partial \Omega}}^2, \qquad f \in H_\diamond^r(\partial \Omega),
        \end{equation*}
        defines an equivalent norm in $H^r_\diamond(\partial \Omega)$ for any $-\tfrac{1}{2} \leq r \leq \tfrac{1}{2}$ \cite[Appendix~B]{Garde_2019c}. The corresponding inner product for $H^r_\diamond(\partial \Omega)$ is
        \begin{equation*} 
	        \inner{f, g}_r = \sum_{k=1}^\infty \lambda_k^{-2r} \inner{f, \phi_k}_{\partial \Omega}\inner{\phi_k, g}_{\partial \Omega}
        \end{equation*}
        for $f,g \in H^r_\diamond(\partial \Omega)$.

        Let $\mathcal{I}_s: H^{s}_\diamond(\partial \Omega) \to  L^2_\diamond(\partial \Omega)$ be the embedding of $H^{s}_\diamond(\partial \Omega)$ into $L^2_\diamond(\partial \Omega)$. It is straightforward to check that $\psi_{j,s} = \lambda_j^{s} \phi_j$, $j \in \N$, form an orthonormal basis for $H^{s}_\diamond(\partial \Omega)$ with respect to the inner product $\langle \, \cdot \, ,  \, \cdot \,\rangle_s$ and that $\mathcal{I}_s$ is characterised by the singular value decomposition
        \begin{equation*}
        	\mathcal{I}_s: \psi_{j,s} \mapsto \lambda_j^{s} \phi_j, \qquad j \in \N.
        \end{equation*}
        Hence,
        \begin{equation} \label{eq:embed_schatten}
        	\norm{\mathcal{I}_s}_{\mathscr{L}_q(H^{s}_\diamond(\partial \Omega), L^2_\diamond(\partial \Omega))}^q = \sum_{j=1}^\infty (\lambda_j^{s})^q < \infty, \qquad q > \frac{1}{s},
        \end{equation}
        by virtue of \eqref{eq:LSchatten}, and thus $\mathcal{I}_s \in \mathscr{L}_q(H^{s}_\diamond(\partial \Omega), L^2_\diamond(\partial \Omega))$ for $q > 1/s$.  

		The dual operator $\mathcal{I}'_s: L^2_\diamond(\partial \Omega) \to H^{-s}_\diamond(\partial \Omega)$ is itself the embedding of $L^2_\diamond(\partial \Omega)$ into $H^{-s}_\diamond(\partial \Omega)$. In particular,
        \begin{equation*}
        	L|_{L^2_\diamond(\partial \Omega)} = \mathcal{I}_s\, L\, \mathcal{I}_s'.
        \end{equation*}
        Due to the H\"older inequality for the Schatten norms \cite[Theorem~7.8(b)]{Weidmann1980}, we finally have 
        \begin{align*}
            \norm{L}_{\mathscr{L}_p(L^2_\diamond(\partial \Omega))} &\leq C \norm{\mathcal{I}_s}_{\mathscr{L}_{2p}(H^{s}_\diamond(\partial \Omega), L^2_\diamond(\partial \Omega))} \norm{L \mathcal{I}_s'}_{\mathscr{L}_{2p}(L^2_\diamond(\partial \Omega), H^s_\diamond(\partial \Omega))} \\
            &\leq C \norm{\mathcal{I}_s}_{\mathscr{L}_{2p}(H^{s}_\diamond(\partial \Omega), L^2_\diamond(\partial \Omega))}^2  \norm{L}_{\mathscr{L}(H^{-s}_\diamond(\partial \Omega), H^s_\diamond(\partial \Omega))},
        \end{align*}
        where we also used \cite[Theorem~7.8(c)]{Weidmann1980} and the fact that the Schatten norms of an operator and its dual obviously coincide. The assertion now follows from \eqref{eq:embed_schatten}.
    \end{proof}

    Under the assumptions of Section~\ref{sec:continuum} and if $\Gamma = \partial \Omega$, $\Lambda(A+B) \in \mathscr{L}(H^{-1/2}_{\diamond}(\partial \Omega), H^{1/2}_{\diamond}(\partial \Omega))$ for small enough $B$, and furthermore the Taylor series \eqref{eq:Lseries} converges in $\mathscr{L}(H^{-1/2}_{\diamond}(\partial \Omega), H^{1/2}_{\diamond}(\partial \Omega))$. These facts follow straightforwardly as  $T: H^1_\diamond(\Omega) \to H^{1/2}_\diamond(\partial \Omega)$ and $N(A): H^{-1/2}_\diamond(\partial \Omega) \to H^1_\diamond(\Omega)$ are bounded, which allows for considering the finer topology of $\mathscr{L}(H^{-1/2}_{\diamond}(\partial \Omega), H^{1/2}_{\diamond}(\partial \Omega))$ instead of that of $\mathscr{L}(L^2_{\diamond}(\partial \Omega))$; see Remark~\ref{rm:quotient}. Under the assumptions of Lemma~\ref{lemma:LHS} on $\Omega$ and $\Gamma$, one may thus consider $\mathscr{L}_{2}(\mathcal{X}) \subset \mathscr{L}_{\infty}(\mathcal{X}) = \mathscr{L}_{\rm C}(\mathcal{X}) $ in place of both $\mathscr{L}(\mathcal{X})$ and $\mathscr{L}_{\rm C}(\mathcal{X})$ in Section~\ref{sec:reversion} and, in particular, the projection $Q$ can be chosen systematically as the orthogonal projection onto $\mathcal{Y}$ in $\mathscr{L}_{2}(\mathcal{X})$.
        
	\bibliographystyle{plain}

\begin{thebibliography}{10}
		
		\bibitem{Alberti2019}
		G.~S. Alberti and M.~Santacesaria.
		\newblock Calder{\'{o}}n's inverse problem with a finite number of
		measurements.
		\newblock {\em Forum Math. Sigma}, 7:e35, 2019.
		
		\bibitem{Alberti2020}
		G.~S. Alberti and M.~Santacesaria.
		\newblock Calder{\'{o}}n's inverse problem with a finite number of measurements
		{II}: independent data.
		\newblock {\em Appl. Anal.}, 2020.
		\newblock To appear.
		
		\bibitem{Alberti2021}
		G.~S. Alberti and M.~Santacesaria.
		\newblock Infinite dimensional compressed sensing from anisotropic measurements
		and applications to inverse problems in {PDE}.
		\newblock {\em Appl. Comput. Harmon. A.}, 50:105--146, 2021.
		
		\bibitem{Alessandrini2017}
		G.~Alessandrini, M.~V. de~Hoop, and R.~Gaburro.
		\newblock Uniqueness for the electrostatic inverse boundary value problem with
		piecewise constant anisotropic conductivities.
		\newblock {\em Inverse Problems}, 33(12), 2017.
		\newblock Article ID 125013.
		
		\bibitem{Alessandrini2018}
		G.~Alessandrini, M.~V. de~Hoop, R.~Gaburro, and E.~Sincich.
		\newblock {EIT} in a layered anisotropic medium.
		\newblock {\em Inverse Probl. Imag.}, 12(3):667--676, 2018.
		
		\bibitem{Arridge_2012}
		S.~Arridge, S.~Moskow, and J.~C. Schotland.
		\newblock Inverse {B}orn series for the {C}alderon problem.
		\newblock {\em Inverse Problems}, 28(3), 2012.
		\newblock Article ID 035003.
		
		\bibitem{Astala2006a}
		K.~Astala and L.~P{\"a}iv{\"a}rinta.
		\newblock {C}alder\'on's inverse conductivity problem in the plane.
		\newblock {\em Ann. Math.}, 163(1):265--299, 2006.
		
		\bibitem{Astala2005}
		K.~Astala, L.~P{\"a}iv{\"a}rinta, and M.~Lassas.
		\newblock Calder\'on's inverse problem for anisotropic conductivity in the
		plane.
		\newblock {\em Comm. PDE}, 30(1--2):207--224, 2005.
		
		\bibitem{Borcea2002a}
		L.~Borcea.
		\newblock Electrical impedance tomography.
		\newblock {\em Inverse Problems}, 18(6):R99--R136, 2002.
		
		\bibitem{Borcea2002}
		L.~Borcea.
		\newblock Addendum to ``{E}lectrical impedance tomography".
		\newblock {\em Inverse Problems}, 19(4):997--998, 2003.
		
		\bibitem{Calderon1980}
		A.~P. {C}alder{\'o}n.
		\newblock On an inverse boundary value problem.
		\newblock In {\em Seminar on {N}umerical {A}nalysis and its {A}pplications to
			{C}ontinuum {P}hysics}, pages 65--73. Soc. Brasil. Mat., Rio de Janeiro,
		1980.
		
		\bibitem{Garde2020c}
		V.~Candiani, J.~Dard\'e, H.~Garde, and N.~Hyv{\"o}nen.
		\newblock Monotonicity-based reconstruction of extreme inclusions in electrical
		impedance tomography.
		\newblock {\em SIAM J. Math. Anal.}, 52(6):6234--6259, 2020.
		
		\bibitem{CaroRogers2016}
		P.~Caro and K.~M. Rogers.
		\newblock Global uniqueness for the {C}alder{\'o}n problem with {L}ipschitz
		conductivities.
		\newblock {\em Forum Math. Pi}, 4:e2, 2016.
		
		\bibitem{Cheney1999}
		M.~Cheney, D.~Isaacson, and J.~C. Newell.
		\newblock Electrical impedance tomography.
		\newblock {\em SIAM Review}, 41(1):85--101, 1999.
		
		\bibitem{Cheng89}
		K.-S. Cheng, D.~Isaacson, J.~S. Newell, and D.~G. Gisser.
		\newblock Electrode models for electric current computed tomography.
		\newblock {\em IEEE Trans. Biomed. Eng.}, 36:918--924, 1989.
		
		\bibitem{Darde21}
		J.~Dard\'{e}, N.~Hyv\"{o}nen, T.~Kuutela, and T.~Valkonen.
		\newblock Electrodeless electrode model for electrical impedance tomography.
		\newblock {\em Preprint arXiv:2102.01926}, 2021.
		
		\bibitem{Engl1996}
		H.~W. Engl, M.~Hanke, and A.~Neubauer.
		\newblock {\em Regularization of inverse problems}.
		\newblock Kluwer Academic Publishers, 1996.
		
		\bibitem{Fernandes1997}
		P.~Fernandes and G.~Gilardi.
		\newblock Magnetostatic and electrostatic problems in inhomogeneous anisotropic
		media with irregular boundary and mixed boundary conditions.
		\newblock {\em Math. Models Methods Appl. Sci.}, 7(7):957--991, 1997.
		
		\bibitem{Ferreira2009}
		D.~Ferreira, C.~Kenig, J.~Sj{\"o}strand, and G.~Uhlmann.
		\newblock On the linearized local {C}alder{\'{o}}n problem.
		\newblock {\em Math. Res. Lett.}, 16(6):955--970, 2009.
		
		\bibitem{Garde_2019b}
		H.~Garde.
		\newblock Reconstruction of piecewise constant layered conductivities in
		electrical impedance tomography.
		\newblock {\em Comm. PDE}, 45(9):1118--1133, 2020.
		
		\bibitem{Garde2020}
		H.~Garde and N.~Hyv\"onen.
		\newblock Optimal depth-dependent distinguishability bounds for electrical
		impedance tomography in arbitrary dimension.
		\newblock {\em SIAM J. Appl. Math.}, 80(1):20--43, 2020.
		
		\bibitem{GardeHyvonen2021}
		H.~Garde and N.~Hyv{\"o}nen.
		\newblock Mimicking relative continuum measurements by electrode data in
		two-dimensional electrical impedance tomography.
		\newblock {\em Numer. Math.}, 147(3):579--609, 2021.
		
		\bibitem{Garde_2019c}
		H.~Garde, N.~Hyv\"onen, and T.~Kuutela.
		\newblock On regularity of the logarithmic forward map of electrical impedance
		tomography.
		\newblock {\em SIAM J. Math. Anal.}, 52(1):197--220, 2020.
		
		\bibitem{GardeStaboulis_2016}
		H.~Garde and S.~Staboulis.
		\newblock Convergence and regularization for monotonicity-based shape
		reconstruction in electrical impedance tomography.
		\newblock {\em Numer. Math.}, 135(4):1221--1251, 2017.
		
		\bibitem{Hanke03}
		M.~Hanke and M.~Br\"{u}hl.
		\newblock Recent progress in electrical impedance tomography.
		\newblock {\em Inverse Problems}, 19(6):S65--S90, 2003.
		
		\bibitem{Harrach_2019}
		B.~Harrach.
		\newblock Uniqueness and {L}ipschitz stability in electrical impedance
		tomography with finitely many electrodes.
		\newblock {\em Inverse Problems}, 35(2), 2019.
		\newblock Article ID 024005.
		
		\bibitem{Harrach2021}
		B.~Harrach.
		\newblock An {I}ntroduction to {F}inite {E}lement {M}ethods for {I}nverse
		{C}oefficient {P}roblems in {E}lliptic {PDE}s.
		\newblock {\em {J}ahresber. {D}tsch. {M}ath.}, 123(3):183--210, 2021.
		
		\bibitem{Harrach10}
		B.~Harrach and J.~K. Seo.
		\newblock Exact shape-reconstruction by one-step linearization in electrical
		impedance tomography.
		\newblock {\em SIAM J. Math. Anal.}, 42(4):1505--1518, 2010.
		
		\bibitem{Harrach13}
		B.~Harrach and M.~Ullrich.
		\newblock Monotonicity-based shape reconstruction in electrical impedance
		tomography.
		\newblock {\em SIAM J. Math. Anal.}, 45(6):3382--3403, 2013.
		
		\bibitem{Hyvonen2004}
		N.~Hyv{\"o}nen.
		\newblock Complete electrode model of electrical impedance tomography:
		Approximation properties and characterization of inclusions.
		\newblock {\em SIAM J. Appl. Math.}, 64(3):902--931, 2004.
		
		\bibitem{Hyvonen09}
		N.~Hyv{\"o}nen.
		\newblock Approximating idealized boundary data of electric impedance
		tomography by electrode measurements.
		\newblock {\em Math. Models Methods Appl. Sci.}, 19(7):1185--1202, 2009.
		
		\bibitem{Hyvonen2017}
		N.~Hyv{\"o}nen and L.~Mustonen.
		\newblock Smoothened complete electrode model.
		\newblock {\em SIAM J. Appl. Math.}, 77(6):2250--2271, 2017.
		
		\bibitem{Hyvonen_2018}
		N.~Hyv\"onen, L.~P\"aiv\"arinta, and J.~P. Tamminen.
		\newblock Enhancing {D}-bar reconstructions for electrical impedance tomography
		with conformal maps.
		\newblock {\em Inverse Probl. Imag.}, 12(2):373--400, 2018.
		
		\bibitem{Imanuvilov2010}
		O.~Y. Imanuvilov, G.~Uhlmann, and M.~Yamamoto.
		\newblock The {C}alder{\'o}n problem with partial data in two dimensions.
		\newblock {\em J. Amer. Math. Soc.}, 23(3):655--691, 2010.
		
		\bibitem{Imanuvilov_2015}
		O.~Y. Imanuvilov, G.~Uhlmann, and M.~Yamamoto.
		\newblock The {N}eumann-to-{D}irichlet map in two dimensions.
		\newblock {\em Adv. Math.}, 281:578--593, 2015.
		
		\bibitem{Isakov2007}
		V.~Isakov.
		\newblock On uniqueness in the inverse conductivity problem with local data.
		\newblock {\em Inverse Probl. Imag.}, 1:95--105, 2007.
		
		\bibitem{Kenig_2013}
		C.~Kenig and M.~Salo.
		\newblock The {C}alder{\'{o}}n problem with partial data on manifolds and
		applications.
		\newblock {\em Anal. {PDE}}, 6(8):2003--2048, 2013.
		
		\bibitem{Kenig_2014}
		C.~Kenig and M.~Salo.
		\newblock Recent progress in the {C}alder{\'{o}}n problem with partial data.
		\newblock {\em Contemp. Math.}, 615:193--222, 2014.
		
		\bibitem{Kohn1985}
		R.~Kohn and M.~Vogelius.
		\newblock Determining conductivity by boundary measurements {II}. {I}nterior
		results.
		\newblock {\em Comm. Pure Appl. Math.}, 38(5):643--667, 1985.
		
		\bibitem{Lechleiter2008}
		A.~Lechleiter and A.~Rieder.
		\newblock Newton regularizations for impedance tomography: convergence by local
		injectivity.
		\newblock {\em Inverse Problems}, 24(6), 2008.
		\newblock Article ID 065009.
		
		\bibitem{Nachman1996}
		A.~I. Nachman.
		\newblock Global uniqueness for a two-dimensional inverse boundary value
		problem.
		\newblock {\em Ann. Math.}, 143:71--96, 1996.
		
		\bibitem{Nachman2010}
		A.~I. Nachman and B.~Street.
		\newblock Reconstruction in the {C}alder{\'o}n problem with partial data.
		\newblock {\em Comm. {PDE}}, 35(2):375--390, 2010.
		
		\bibitem{Pommerenke1992}
		C.~Pommerenke.
		\newblock {\em Boundary behaviour of conformal maps}, volume 299 of {\em
			Grundlehren der Mathematischen Wissenschaften}.
		\newblock Springer-Verlag, Berlin, 1992.
		
		\bibitem{Somersalo1992}
		E.~Somersalo, M.~Cheney, and D.~Isaacson.
		\newblock Existence and uniqueness for electrode models for electric current
		computed tomography.
		\newblock {\em SIAM J. Appl. Math.}, 52(4):1023--1040, 1992.
		
		\bibitem{Sylvester1990}
		J.~Sylvester.
		\newblock An anisotropic inverse boundary value problem.
		\newblock {\em Comm. Pure Appl. Math.}, 43(2):201--232, 1990.
		
		\bibitem{Sylvester1987}
		J.~Sylvester and G.~Uhlmann.
		\newblock A global uniqueness theorem for an inverse boundary value problem.
		\newblock {\em Ann. Math.}, 125:153--169, 1987.
		
		\bibitem{Thorp1960}
		E.~O. Thorp.
		\newblock Projections onto the subspace of compact operators.
		\newblock {\em Pac. J. Math.}, 10(2):693--696, 1960.
		
		\bibitem{Uhlmann2009}
		G.~Uhlmann.
		\newblock Electrical impedance tomography and {C}alder\'on's problem.
		\newblock {\em Inverse Problems}, 25(12), 2009.
		\newblock Article ID 123011.
		
		\bibitem{Valent1988}
		T.~Valent.
		\newblock {\em Boundary Value Problems of Finite Elasticity}.
		\newblock Springer New York, 1988.
		
		\bibitem{Weidmann1980}
		J.~Weidmann.
		\newblock {\em Linear operators in {H}ilbert spaces}, volume~68 of {\em
			Graduate Texts in Mathematics}.
		\newblock Springer-Verlag, New York-Berlin, 1980.
		
	\end{thebibliography}

\end{document}